\documentclass [11pt]{amsart}
\usepackage {amsmath, amssymb, amscd, mathrsfs, url,pinlabel,hyperref, verbatim}
\usepackage[text={5.5in,9in},centering,letterpaper,dvips]{geometry}
\usepackage{color,multirow,dcpic,latexsym,pictexwd,graphicx,epstopdf}
\usepackage{verbatim, relsize, latexsym, latexsym, amsthm,tikz, mdframed, float, mathtools, dsfont, flowchart, multicol, stmaryrd, caption, capt-of}

\usetikzlibrary{positioning, arrows.meta}
\tikzstyle{bag} = [align=center]

\newtheorem {theorem}{Theorem}
\newtheorem {lemma}[theorem]{Lemma}
\newtheorem {proposition}[theorem]{Proposition}
\newtheorem {corollary}[theorem]{Corollary}

\newtheorem {definition}[theorem]{Definition}

\theoremstyle{remark}

\newtheorem {remark}[theorem]{Remark}

\numberwithin{equation}{section}
\numberwithin{theorem}{section}

\newcommand\Z{\mathbb{Z}}

\title[ASL Graphs, Homology Cobordism and Connected HF Homology]{Almost Simple Linear Graphs, Homology Cobordism and Connected Heegaard Floer Homology}

\author{\c{C}a\u{g}r{\i} Karakurt}
\address{Department of Mathematics, Bo\u{g}az{\i}\c{c}{\i}  University, Bebek, 34342 Istanbul, Turkey.}
\email{\href{mailto:cagri.karakurt@boun.edu.tr}{cagri.karakurt@boun.edu.tr}}
\urladdr{\url{http://web0.boun.edu.tr/cagri.karakurt/}}

\author{O{\u{g}}uz \c{S}avk}
\address{Department of Mathematics, Bo\u{g}az{\i}\c{c}{\i}  University, Bebek, 34342 Istanbul, Turkey and Max Planck Institut f\"{u}r Mathematik, Vivatsgasse 7, 53111 Bonn, Germany.}
\email{\href{mailto:oguz.savk@boun.edu.tr}{oguz.savk@boun.edu.tr}, \href{mailto:savk@mpim-bonn.mpg.de}{savk@mpim-bonn.mpg.de}}
\urladdr{\url{https://sites.google.com/view/oguzsavk/}}
\date{}

\begin{document}

\begin{abstract}
Continuing our previous work in \cite{KS20}, we effectively compute connected Heegaard Floer homologies of two families of Brieskorn spheres realized as the boundaries of almost simple linear graphs. Using Floer theoretic invariants of Dai, Hom, Stoffregen, and Truong, we show that these Brieskorn spheres also generate infinite rank summands in the homology cobordism group. Our computations also have applications to the concordance of classical knots and $0$-concordance of $2$-knots.
\end{abstract}
\maketitle

\section{Introduction}
Since the 1980s, the homology cobordism group $\Theta^3_\mathbb{Z}$ (see Section~\ref{homcobordism}) has been a central object in low-dimensional topology \cite{M18}. Although its algebraic structure is still mysterious, the recent achievement of Dai, Hom, Stoffregen, and Truong indicates that it is very complicated. Previously, it is known that $\Theta^3_\mathbb{Z}$ has infinitely generated subgroups $\mathbb{Z}^\infty$ \cite{F90}, \cite{FS90} and infinite summands $\mathbb{Z}$ \cite{Fro02}, \cite{OS03b}. 

\begin{theorem}[Theorem~1.1, \cite{DHST18}]
\label{DHST}
Let $\{X_n = \Sigma(2n+1,4n+1,4n+3) \}^\infty_{ n=1}$ be a family of Brieskorn spheres. Then the homology cobordism group $\Theta^3_\mathbb{Z}$ has an infinite rank summand $\Z^\infty$ generated by the homology cobordism classes $[X_n]^\infty_{ n=1} $.
\end{theorem} 

The elegant proof of Dai, Hom, Stoffregen, and Truong relies on the construction and computation of a linearly independent set of invariants of $\Theta^3_\mathbb{Z}$ originated from involutive Heegaard Floer homology, introduced by Hendricks and Manolescu \cite{HM17}. Their proof also subsumes the several techniques in the theoretical continuation of involutive Heegaard Floer homology \cite{DM17}, \cite{DS17}, \cite{HMZ18} and \cite{HHL18}.  

The manifolds mentioned in Theorem~\ref{DHST} are carefully selected in the sense that they have sufficiently complicated graded roots and arbitrarily large Ozsv\'ath-Szab\'o $d$-invariants. Even though classical and involutive Heegaard Floer homologies are known to be not homology cobordism invariants, their refinement \emph{connected Heegaard Floer homology}, defined by Hendricks, Hom, and Lidman, turns out to be an invariant of $\Theta^3_\mathbb{Z}$. It essentially detects the sufficient complexity of graded roots and corresponds to the homology of monotone subroots of graded roots \cite{HHL18}.

In this article, we efficiently compute connected Heegaard Floer homologies of two families of Brieskorn spheres $\Sigma(p,q,r)$ realized as the boundaries of \emph{almost simple linear graphs}, see Section~\ref{aslgraphs}. Here, $p,q$ and $r$ are chosen as pairwise coprime positive integers with $pq+pr-qr =1$. Clearly, parameters of $X_n$ satisfy the latter identity.

\begin{theorem}
\label{main}
Consider the following two families of Brieskorn spheres: $$\{Y_n = \Sigma(2n+1,3n+2,6n+1) \}^\infty_{ n=1} \ \ \text{and} \ \ \{Z_n = \Sigma(2n+1,3n+1,6n+5) \}^\infty_{ n=1}.$$ Then their monotone subroots are shown in Figure~\ref{fig:monotoneyz}. Therefore, their connected Heegaard Floer homologies are given in Table~\ref{tab:chfl}.
\end{theorem}
\begin{figure}[htbp]
\begin{center}
\includegraphics[width=1\columnwidth]{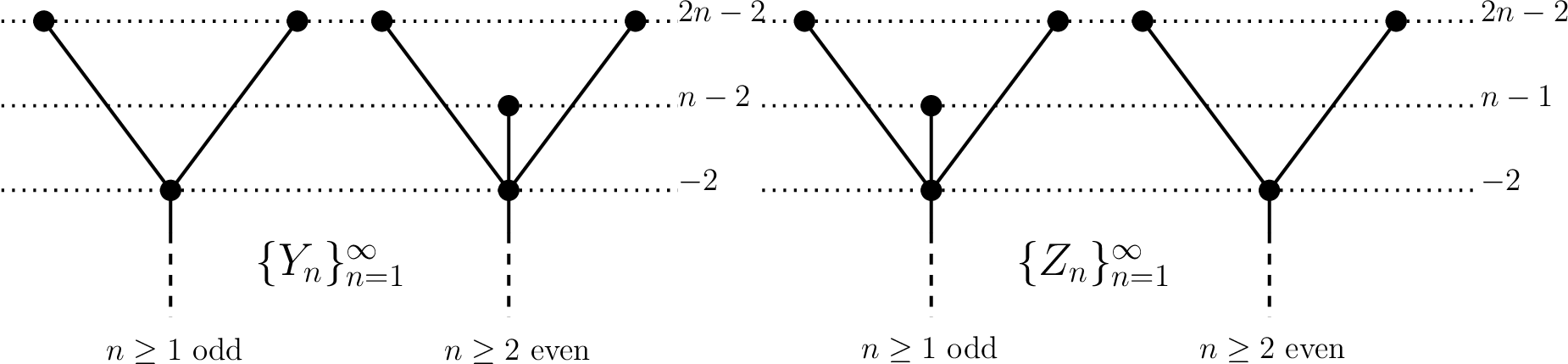}
\end{center}
\caption{Monotone subroots of $\{Y_n \}^\infty_{ n=1}$ and $\{Z_n \}^\infty_{ n=1}$.}
\label{fig:monotoneyz}
\end{figure} 
\begin{table}[htbp]
\begin{tabular}{|l|l|l|l|l|}
\hline
$HF_{\mathrm{conn}}$ & $n \geq 1$ odd & $n \geq 2$ even\\ \hline
$\{Y_n\}$ & $\mathbb{F}_{(2n-2))} (n)$ & $\mathbb{F}_{(2n-2))} (n) \oplus \mathbb{F}_{(n-2))} (n/2)$   \\ \hline
$\{Z_n\}$ & $\mathbb{F}_{(2n-2))} (n) \oplus \mathbb{F}_{(n+1))} ((n+1)/2)$ & $\mathbb{F}_{(2n-2))} (n)$   \\ \hline
\end{tabular}
\vskip\baselineskip
\caption{The connected Heegaard Floer homologies of $\{Y_n \}^\infty_{ n=1}$ and $\{Z_n \}^\infty_{ n=1}$.}
\label{tab:chfl}
\end{table}
Besides the family of Dai, Hom, Stoffregen, and Truong, our additional two families are the boundaries of almost simple linear graphs for odd values of $p$. When $p$ is even, by \cite{KS20}, we always have $d = -2\bar{\mu}$ where $d$ is the Ozsv\'ath-Szab\'o $d$-invariant \cite{OS03a} and $\bar{\mu}$ is the Neumann-Siebenmann $\bar{\mu}$-invariant \cite{N80} and \cite{S80}. This is also the case for involutive correction terms $\underline{d} = \overline{d}$, see \cite{DM17}. In this case, their connected Heegaard Floer homologies are trivial, so they do not provide valuable information for the homology cobordism group $\Theta^3_\Z$. Hence, the interesting families of almost simple linear graphs arise only when $p$ is odd.

For large values of $n$, the full graded roots of the Brieskorn spheres we consider are quite complicated and intractable even with the help of computer programs. On the other hand, monotone subroots are sufficient to determine the corresponding connected Heegaard Floer homology and they have relatively simple descriptions in terms of \emph{packed sequences} that we introduced in Section~\ref{packed}. By the work of N\'emethi \cite{Nem05}, Dai \cite{D18}, and Dai and Manolescu \cite{DM17}, the graded roots are characterized by Laufer sequences and monotone subroots can be parametrized by using $d$- and $\bar{\mu}$-invariants. Moreover, the $\bar{\mu}$-invariant can be read off from a graded root and the $d$-invariant also lies on the graded root. 

In this work, we relate Laufer sequences with $d$-invariants in terms of the Ozsv\'ath-Szab\'o algorithm, see Proposition~\ref{laufergood}. Therefore, we are able to compute the minimal graded subroot of the whole graded root including the monotone subroot for infinite families of plumbings, for a comparison see Remark~\ref{comparison}. 

Our computation procedure is indicated schematically in Figure~\ref{fig:guide}. In historical order, it is a mixture of the work of Ozsv\'ath and Szab\'o \cite{OS03a}, N\'emethi \cite{Nem05}, Can and the first author \cite{CK14}, Dai and Manolescu \cite{DM17}, Öztürk and the first author \cite{KO17}, and Hendricks, Hom, and Lidman \cite{HHL18}, and the authors \cite{KS20}. Note that the initial step in the scheme heavily depends on our previous joint work \cite{KS20} in which we characterized characteristic classes supporting good full paths for almost simple linear graphs and then computed their $d$-invariants.

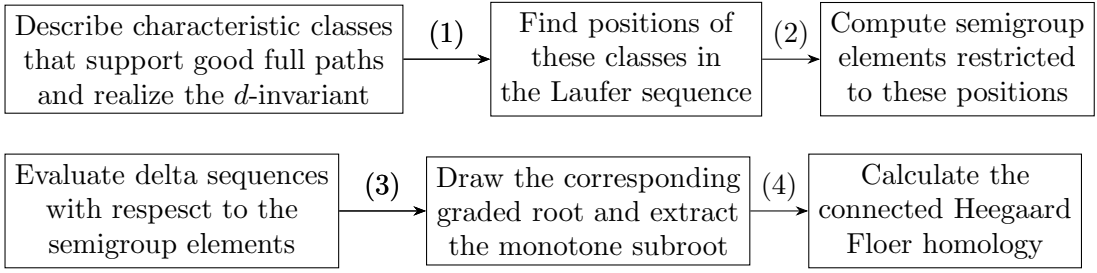
\begin{figure}[htbp]
\centering
\begin{tikzpicture}
        [block/.style={draw,minimum width=#1,minimum height=1em},
        block/.default=10em,high/.style={minimum height=3em},
        node distance=2em, > = Stealth]

        % Nodes
        \node[block=3em,high,right=3em, align=center] (n0) {Describe characteristic classes \\ that support good full paths \\ and realize the $d$-invariant};
        \node[block=3em,high,right=3em of n0, align=center] (n1) {Find positions of \\ these classes in \\ the Laufer sequence};
        \node[block=3em,high,right=of n1, align=center] (n2) {Compute semigroup \\ elements restricted \\ to these positions};
        
        % Connections
        \foreach \i [count=\j from 1] in {0,1}
        \draw[->] (n0) -- node [above] {(1)} (n1);
        \draw[->] (n1) -- node [above] {(2)} (n2); 
\end{tikzpicture}

\vskip\baselineskip

\begin{tikzpicture}
        [block/.style={draw,minimum width=#1,minimum height=1em},
        block/.default=10em,high/.style={minimum height=3em},
        node distance=2em, > = Stealth]

        % Nodes
        \node[block=3em,high,right=3em, align=center] (n0) {Evaluate delta sequences \\ with respesct to the \\ semigroup elements};
        \node[block=3em,high,right=3em of n0, align=center] (n1) {Draw the corresponding \\ graded root and extract \\ the monotone subroot};
        \node[block=3em,high,right=of n1, align=center] (n2) {Calculate the \\ connected Heegaard \\ Floer homology};
        
        % Connections
        \foreach \i [count=\j from 1] in {0,1}
        \draw[->] (n0) -- node [above] {(3)} (n1);
        \draw[->] (n1) -- node [above] {(4)} (n2); 
\end{tikzpicture}
\caption{User's guide to our computation procedure.}
\label{fig:guide}
\end{figure}

Using the involutive Floer theoretic invariants of Dai, Hom, Stoffregen, and Truong, we present new families of homology spheres generating infinite rank summands in the homology cobordism group. With the help of gauge theoretic invariants, we also completely compare their homology cobordism classes.

\begin{theorem}
\label{cobordism}
The homology cobordism classes $[Y_n]^\infty_{ n=1} $ and $[Z_n]^\infty_{n=1} $ also generate $\mathbb{Z}^\infty$ summands in $\Theta^3_\mathbb{Z}$. In particular, the families $\{X_n \}^\infty_{ n=1}$, $\{Y_n \}^\infty_{ n=1}$ and $\{Z_n \}^\infty_{ n=1}$ are not homology cobordant to each other with an exception $X_1=\Sigma(3,5,7)=Y_1$.
\end{theorem}

In Section~\ref{Xn}, we easily recompute the monotone subroot of $\{X_n \}^\infty_{ n=1}$. From this data, it is straightforward to compute its connected Heegaard Floer homology, consult the proof of Theorem~\ref{main}. It is worth to compare our computation for $\{X_n \}^\infty_{ n=1}$ with \cite{HKL16}, \cite{St17} and \cite{DS17}. 

We may regard Brieskorn spheres as Seifert fibered spheres with three singular fibers. Recently, Seetharaman, Yue, and Zhu studied the invariance of $d$-invariants and monotone subroots of Seifert fibered spheres under Seifert fiber surgeries, see \cite{SYZ21}. Together with their result, we have the following generalization.

\begin{corollary}
\label{invariance}
Consider the following three families of Brieskorn spheres:
\begin{align*}
& \{ X_{n,m} = \Sigma(2n+1,4n+1, 4n+3 + 2m(2n+1)(4n+1)) \}^\infty_{ n,m=1}, \\
& \{ Y_{n,m} = \Sigma(2n+1,3n+2, 6n+1 + 2m(2n+1)(3n+2)) \}^\infty_{ n,m=1}, \\
& \{ Z_{n,m} = \Sigma(2n+1,3n+1, 6n+5 + 2m(2n+1)(3n+1)) \}^\infty_{ n,m=1} .
\end{align*}
Then homology cobordism classes $[ X_{n,m} ]^\infty_{n,m = 1}$, $[ Y_{n,m} ]^\infty_{n,m = 1}$, and $[ Z_{n,m} ]^\infty_{n,m = 1}$ also generate $\mathbb{Z}^\infty$ summands in $\Theta^3_\mathbb{Z}$.
\end{corollary}

\begin{proof}
If we perform twice $(-1)$-surgeries along the singular fibers of degree $4n+1$, $6n+1$ and $6n+5$ of $X_n$, $Y_n$ and $Z_n$ respectively and if we repeat this process inductively, then we obtain families $X_{n,m}$, $Y_{n,m}$ and $Z_{n,m}$ for $n,m \geq 1$. Then the claim is a conclusion of Proposition~\ref{monotonesubroot3}, Theorem~\ref{DHST}, Theorem~\ref{cobordism}, and \cite[Theorem~1.3]{SYZ21}. 
\end{proof}

Applying double branched coverings, one can pass from pretzel knots $P(-p,q,r)$ to Brieskorn spheres $\Sigma(p,q,r)$. Thus we have an analogous result for the concordance group of knots $\mathcal{C}$ and for its subgroup generated by topologically slice knots $\mathcal{C}_\mathrm{TS}$, compare with \cite{OSS17} and \cite{DHST19}.

\begin{corollary}
\label{concordance}
Consider the following two families of pretzel knots: $$\{K_n = P(-(2n+1),3n+2,6n+1) \}^\infty_{ n=1} \ \text{and} \ \{J_n = P(-(2n+1),3n+1,6n+5) \}^\infty_{ n=1}.$$ Then the concordance classes $[K_n]^\infty_{ n=1}$ and $[J_n]^\infty_{ n=1}$ also generate $\mathbb{Z}^\infty$ summands in $\mathcal{C}$. In particular, they generate $\mathbb{Z}^\infty$ summands in $\mathcal{C}_\mathrm{TS}$ for odd values of $n$.
\end{corollary}

\begin{proof}[Proof of Corollary~\ref{concordance}]
Since the invariants $\{ \phi_j \}_{j \in \mathbb{N}}$ of Dai, Hom, Stoffregen and Truong are also spin homology cobordism invariants \cite[Remark 1.3]{DHST18}. By taking the double branched covers and using Theorem~\ref{cobordism}, we conclude the result. The second part of our statement follows from the result of Cochran and Lickorish \cite{CL86} since they have the trivial Alexander polynomial, so they are topologically slice due to Freedman \cite{F82}.
\end{proof}

\begin{remark}
Using the work of Furuta \cite{F90}, Endo showed that pretzel knots in Corollary~\ref{concordance} generate $\mathbb{Z}^\infty$ subgroups in $\mathcal{C}_\mathrm{TS}$ \cite{E95}.
\end{remark}

Let $\mathcal{QA}$ and $\mathcal{T}$ denote subgroups of $\mathcal{C}$ generated by quasi-alternating knots and torus knots. Using the result of Alfieri, Kang, and Stipsicz \cite{AKS19}, we see that the complexity of our knots remains the same even if we mod out the concordance group of knots by the subgroups of the quasi-alternating and torus knots.

\begin{corollary}
\label{concordance2}
The concordance classes $[K_n]^\infty_{ n=1}$ and $[J_n]^\infty_{ n=1}$ generate $\mathbb{Z}^\infty$ summands in $\mathcal{C} / (\mathcal{QA} + \mathcal{T})$. 
\end{corollary}

\begin{proof}
We knot that the connected Heegaard Floer homologies of quasi-alternating knots and torus knots are trivial by \cite[Theorem 1.4 and Theorem 1.5]{AKS19}. The rest of the claim follows from Theorem~\ref{cobordism} as in Corollary~\ref{concordance}.
\end{proof}

Since the concordance group of $2$-knots in $S^4$ is trivial due to Kervaire \cite{K65}, Melvin proposed to study the restricted notion called $0$-concordance and showed that such $2$-knots form a commutative monoid $\mathcal{M}_0$ \cite{M77}. Following the work of Sunukjian \cite{S21} and Dai and Miller \cite{DM19}, we present other families of $2$-knots linearly independent in $\mathcal{M}_0$.

\begin{corollary}
\label{concordance3}
Let $\{ P_n \}^\infty_{ n=1}$ and $\{ S_n \}^\infty_{ n=1}$ denote the boundaries of $2$-twist spins of $\{ K_n \}^\infty_{ n=1}$ and $\{ J_n \}^\infty_{ n=1}$, respectively. Then $0$-concordance classes $[P_n]^\infty_{ n=1}$ and $[S_n]^\infty_{ n=1}$ generate $(\mathbb{Z}^{\geq 0})^\infty$ submonoids in $\mathcal{M}_0$. 
\end{corollary}

\begin{proof}
By the work of Zeeman \cite{Z65}, we see that Brieskorn spheres are Seifert solids for the $2$-twist spin of given pretzel knots and the corresponding $2$-knots bound relevant Brieskorn spheres respectively. Using \cite[Theorem~1.1]{DM19} and Theorem~\ref{cobordism}, we obtain the desired result.
\end{proof}

\begin{remark}
The results about the concordance of knots can be generalized in the sense of Corollary~\ref{invariance}. Also, the rest of the homology cobordism classes and all concordance classes can be compared from the perspective of Theorem~\ref{cobordism}.
\end{remark}

\begin{remark}
Despite the algebraic and geometric complexity of $\Theta^3_\Z$, homology spheres homology cobordant to $S^3$ are also in abundance, see \cite{Sa21} and references therein.
\end{remark}

\subsection*{Organization}
In Section~\ref{plumbingsandroots} and Section~\ref{Floer}, we describe essential preliminaries about plumbings, graded roots, and Heegaard Floer homology theories. Next, we introduce packed sequences in order to analyze the semigroups of Brieskorn spheres and hence their delta sequences in Section~\ref{packed}. Then will compute their corresponding graded roots and extract monotone subroots. Finally, we respectively prove Theorem~\ref{main} and Theorem~\ref{cobordism}.

\subsection*{Acknowledgements}
The authors are grateful to Irving Dai and Masaki Taniguchi for their helpful and valuable comments and suggestions. The final part of this work was done while \c{C}K has visited UT Austin and O\c{S} has visited MPIM. We would like to thank both institutions for their hospitality and support. ÇK is supported by the Fulbright visiting scholar grant and T\"UB\.ITAK B\.IDEB-2219.

\section{Plumbings and Graded Roots}
\label{plumbingsandroots}

\subsection{Plumbings}
\label{plumbing}
Let $\mathcal{J}$ be an index set. Let $G$ be a \emph{plumbing graph}, i.e., a weighted connected tree with vertices $v_j$ and weights $e_j$ for $j \in \mathcal{J}$. Using the graph $G$, we can construct a $4$-manifold $X(G)$ with a boundary $Y(G)$ in the following way: For each $v_j$, assign a $D^2$-bundle over $S^2$ whose Euler number is $e_j$ and plumb two of these $D^2$-bundles if there is an edge connecting the vertices.  

The second homology group $H_2(X(G); \mathbb{Z})$ is generated by the fundamental classes of the zero-sections of $D^2$-bundles, so we have a generator of $H_2(X(G); \mathbb{Z})$ for each vertex of $G$. Thus the intersection form on $H_2(X(G); \mathbb{Z})$ is naturally characterized by the associated intersection matrix $I=(a_{ij})$ with entries given as follows:
\[ a_{ij} = \begin{cases} 
      e_i, & \text{if} \ v_i=v_j, \\
      1, & \text{if} \ v_i \ \text{and} \ v_j \ \text{is connected by one edge}, \\
      0, & \text{otherwise}. 
\end{cases}\]
   
The plumbing graph $G$ is said to be a \emph{negative definite graph} if the corresponding intersection matrix is negative-definite, i.e., $\mathrm{signature}(I)=-|G|$, where $|G|$ denotes the number of vertices of $G$. We call $G$ \emph{unimodular} whenever $\mathrm{det}(I)= \pm 1$. In this case, $Y(G)$ is called a \emph{plumbed homology sphere}.

During the course of our computations, we need the entries of the inverse intersection matrix. There is a fairly simple way for finding entries of $I^{-1}$ due to the work of N\'emethi and  Nicolaescu in \cite[Section 5]{NN05}.

Consider two arbitrary vertices $v$ and $w$ of $G$. Let $I^{-1}_{vw}$ denotes the $(v,w)$-entry of the inverse intersection matrix $I^{-1}$. Since $I$ is negative definite and $G$ is connected, observe that $I^{-1}_{vw} <0$. Also, let $p_{vw}$ be the unique minimal path in $G$ that connects $v$ and $w$. Let $I_{(vw)}$ be the intersection matrix of the complement graph of the path $p_{vw}$. It can be obtained from $I$ deleting all the rows and columns corresponding to the vertices on $p_{vw}$. Then we have $I^{-1}_{vw} = -|\mathrm{det}(I_{(vw)})/ \mathrm{det}(I)|$. For the computation of determinants, we consult \cite[Chapter 5.21]{EN85}. For instance,
\begin{itemize}
\item[(1)] Consider the negative-definite linear graph associated to $\frac{p}{p'} = [t_{11}, t_{12}, \ldots,t_{1n_{1}}]$. Then its determinant is $(-1)^n p$,
\item[(2)] The determinant of the disconnected graph is the product of determinants of its connected components.
\end{itemize}

\subsection{ASL Graphs}
\label{aslgraphs}

Let $p,q$ and $r$ be pairwise relatively prime, ordered, positive integers. A \emph{Brieskorn homology sphere} $$\Sigma(p,q,r) = \{ x^p +y^q +z^r = 0 \} \cap S^5 \subset \mathbb{C}^3$$ can be realized as a boundary of the unimodular negative-definite plumbing graph, for details see \cite[Section~1.1.9]{Sav02}. Further we assume that
\begin{equation}
\label{pqr}
pq+pr-qr=1.
\end{equation} 
\noindent Then the Brieskorn sphere $Y=\Sigma(p,q,r)$ is the boundary of the negative definite, unimodular plumbing graph shown in Figure~ \ref{fig:plumb}. Such graphs are called \emph{almost simple linear graphs} (or ASL-graphs for short) and they were extensively studied by authors in \cite{KS20}.  

\begin{figure}[htbp]
\begin{center}
\includegraphics[width=0.40\columnwidth]{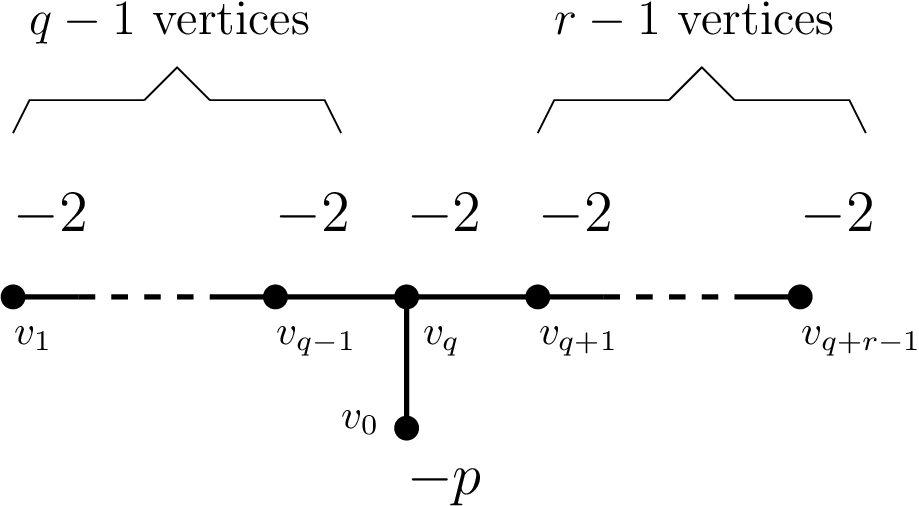}
\end{center}
\caption{The ASL-graphs and the decoration of their vertices.}
\label{fig:plumb}
\end{figure}

\subsection{The Lattices $\mathcal{L}$ and $\mathcal{L}'$, and AR Plumbing Graphs}
\label{AR}

Consider the lattice $\mathcal{L} = H_2(X(G); \mathbb{Z})$. Then we have the following short exact sequence
\begin{equation}
0 \longrightarrow \mathcal{L} \xrightarrow{\mathrm{PD}} \mathcal{L}' \longrightarrow H_1(Y;\mathbb{Z}) \longrightarrow 0  
\end{equation}
where $\mathcal{L}'$ is the dual lattice $\mathrm{Hom}_{\mathbb{Z}}(\mathcal{L},\mathbb{Z})\cong H^2(X;\mathbb{Z})\cong H_2(X,Y;\mathbb{Z})$ given by the Poincar\'e duality and $H_1(Y;\mathbb{Z})$ is trivial whenever $G$ is unimodular.  Here, $\mathrm{PD}(x)$ represent the unique element in $\mathcal{L}'$ which evaluates on each $y \in \mathcal{L}$ as $\langle \mathrm{PD}(x), y \rangle =x^tIy$. And we treat $x$ and $y$ as column matrices and $x^t$ denotes the transpose of $x$. 

Let $\mathcal{J}$ be an index set. We say that $k\in  \mathcal{L}'$ is $\textit{characteristic class}$ if $\langle k, v_j \rangle +e_j = 0 \text{ mod } 2$ holds for every vertex $v_j$ and $j \in \mathcal{J}$. The set of characteristic classes is denoted by $\mathrm{Char}(G)$. The characteristic class $k\in  \mathcal{L}'$ satisfying $\langle k, v_j \rangle = -e_j -2$ for all $j \in \mathcal{J}$ is said to the \emph{canonical class} and it is denoted by $k_{\mathrm{can}}$. For each $k \in \mathrm{Char}(G)$, we define the function
\begin{equation}
\label{eq:chi}
\chi_k: L \to \Z, \ \ \chi_k(v)= -\frac{\langle k, v \rangle + e}{2}. 
\end{equation} 

A plumbing graph $G$ is called \emph{rational} if $\chi(v) \geq 1$ for any $v >0$. Also, $G$ is said to be \emph{almost rational} (or AR for short) if we obtain a rational graph when we decrease the weight $e_0$ of a vertex $v_0$, need not be in a unique way. Note that the plumbing graphs of Brieskorn spheres are always AR, see \cite[Section~8]{Nem05}.

\subsection{Semigroups and Delta Sequences}
\label{delta}

The following numerical sequences were introduced by the first author and Can in \cite{CK14}. They simplify the computation procedure of Heegaard Floer homology in the sense of N\'emethi's graded roots, for details see next section. Here, we recall the definition of delta sequences for Brieskorn spheres.

For a Brieskorn sphere $Y=\Sigma(p,q,r)$, consider the quantity $N_0=pqr-pq-pr-qr$. Let $S_Y$ denote the intersection of the interval $[0,N_0]$ with the numerical \emph{semigroup} generated by the elements $pq$, $pr$ and $qr$. 

Set $Q_Y= \{N_0-s | s \in S_Y\}$ and define $X_Y$ to be the set as a disjoint union of $S_Y$ and $Q_Y$. The pair $(X_Y, \Delta_Y)$ is said to be a \emph{delta sequence} which is recorded by writing the valuation of elements of $S_Y$ and $Q_Y$ with respect to the \emph{delta function} $\Delta_Y$ as an ordered set. The function $\Delta_Y$ is defined as follows: 
\begin{align*}
\Delta_Y: X_Y & \to \{\pm1\} \\
s & \mapsto \Delta_Y(s) = \begin{cases}
+1, \ \text{if} \ s \in S_Y, \\ 
-1, \ \text{if} \ s \in Q_Y.
\end{cases}
\end{align*}

A delta sequence is said to be \emph{reduced} if its consecutive positive values (respectively, consecutive negative values) are written as a single element. We shall denote the reduced delta sequence of $Y$ by $\tilde{\Delta}_Y$. Once $Y$ is fixed, we drop $X_Y$ from the notation and simply write $\Delta_Y$ and $\tilde{\Delta}_Y$ for a delta sequence and a reduced delta sequence respectively. 

\subsection{Graded Roots}
\label{graded}

In \cite[Section~3.2]{Nem05}, N\'emethi introduced the notion of graded roots. They are combinatorial objects carrying the same algebraic information of Heegaard Floer homology of AR plumbings. We can briefly describe their construction in the following fashion.

Let $R$ be an infinite tree with a vertex set $\mathcal{V}_R$. Let $\upsilon: \mathcal{V}_R \rightarrow \Z$ be a grading function satisfying the following properties:
\begin{itemize}
\item[(i)] $\upsilon(x) - \upsilon(y) = \pm 1$ if there is edge connecting $x$ and $y$,
\item[(ii)] $\upsilon(x) > \min\{\upsilon(y), \upsilon(z)\}$ if there are edges from $x$ to $y$ and from $x$ and $z$ with $y \neq z$,
\item[(iii)] $\upsilon$ is bounded below,
\item[(iv)] $\upsilon^{-1}(k)$ is a finite set for every $k$ and $\vert\upsilon^{-1}(k)\vert = 1$ for sufficiently large $k$. 
\end{itemize}

A \emph{graded root} is a pair $(R,\upsilon)$. Once the grading function $\upsilon$ is understood, we drop it from the notation and simply use $R$ for a graded root. To describe a particular graded root, we use the function $\tau: \Z^{\geq 0} \rightarrow \Z$ that is the unique solution of $$\tau(n+1) - \tau(n) = \Delta(n) \ \ \text{with} \ \tau(0)=0.$$ Let $R_n$ be an infinite graph with vertex set $\Z \cap [\tau(n), \infty)$ and edge set $\{ [k,k+1] : k\in [\tau(n), \infty) \}$. For each $n$, we identify all common vertices and edges of $R_n$ and $R_{n+1}$ and we obtain an infinite tree $R$. After the identification, we can assign a grading $\upsilon(x)$ to each vertex $x$, which is a unique integer in any $R_n$ due to the construction.

The \emph{lattice homology} of $(R,\upsilon)$, denoted by $\mathbb{H}^+(R,\upsilon)$, is a combinatorial $\mathbb{F}[U]$-module and it is generated by vertices of $R$ as an $\mathbb{F}$-vector space. The grading of a vertex $x$ is given by $2\upsilon(x)$ and the $U$-action on the vertex $x$ is given by $U.x$ defined to be the sum of vertices $y$ that are connected to $x$ by an edge and satisfy $\upsilon(y) = \upsilon(x)-1$. Here, $U$ is a formal variable that lowers the grading by $2$.

The plus version of Heegaard Floer homology of $Y$ is denoted by $HF^+(Y)$. If $Y$ is a homology sphere, then it is known that we have the following decomposition $$HF^+(Y) \cong \mathcal{T}^+_{d(Y)} \oplus HF_{red}(Y)$$ where $d(Y)$ is the $d$-invariant of $Y$, $\mathcal{T}^+_{d(Y)}$ is a copy of $\mathbb{F}[U,U^{-1}]/U.\mathbb{F}[U]$-module called \emph{tower}, and $HF_{\mathrm{red}}(Y)$ is a finitely generated $\mathbb{F}[U]$-module called \emph{reduced} Heegaard Floer homology. Due to N\'emethi \cite{Nem05}, we have the following isomorphism for an AR plumbed homology sphere $Y(G)$ up to a grading shift:  
\begin{equation}
\label{nemethi}
HF^+(-Y) \cong \mathbb{H}^+(R,\upsilon)\left [(k_{\mathrm{can}}^2 + \vert G \vert) / 4 \right ]
\end{equation}
where $-$ sign indicates the reversed orientation.

\subsection{Monotone Graded Subroots}
\label{monotoneroot}

The concept of monotone graded subroot was appeared in [Section 6, \cite{DM17}].  Let $n$ be a positive integer. Let $h_1, \cdots, h_n$ and $r_1, \cdots, r_n$ be two sequences of rational numbers such that 
\begin{itemize}
\item $h_1 > h_2 > \cdots > h_n$ and $h_i - h_j \equiv 0 \mod 2$ for each $i,j$,
\item $r_1< r_2 < \cdots < r_n$ and $ r_k - r_l \equiv 0 \mod 2 $ for each $k,l$,
\item $h_n \geq r_n$.
\end{itemize}
Let $J_0$ be a natural involution that reflects the vertical axis of the graded root. We can construct the \emph{monotone graded subroot} $$M=M(h_1,r_1; \ldots; h_n,r_n)$$ as follows:

\begin{enumerate}
\item Form the stem of graded root by drawing a single infinite tower with uppermost vertex in degree $r_n$,
\item If $h_n > r_n$, introduce leaves $v_i$ and $J_0v_i$ in degree $h_i$ for each $1 \leq i < n$, 
\item Next, connect $v_i$ and $J_0v_i$ to the stem by using two paths meeting the stem in degree $r_i$ for each $1 \leq i < n$,
\item If $h_n = r_n$, then set $v_n = J_0v_n$ at grading $r_n$ in the second step.
\end{enumerate}

The lattice homology of $M$ can be introduced by following the same process in Section~\ref{graded}.

\subsection{Ozsv\'ath-Szab\'o Algorithm}

Let $\mathcal{J}$ be an index set. Let $G$ be a plumbing graph with vertices $v_j$ for $j \in \mathcal{J}$. A vertex $v_j$ is called a \emph{bad vertex} if $e_j$ is greater than the minus of number of edges of $G$ containing $v_j$. In \cite[Section~3]{OS03a}, Ozsv\'ath and Szab\'o provided an algorithm determining the elements of $\mathrm{Ker}U$ in $\mathbb{K}^+(G)$ for negative-definite plumbing graphs with at most one bad vertices. One can see the next section for notations and the generalization of the Ozsv\'ath-Szab\'o algorithm.

The algorithm runs as follows. Start with an \emph{initial} characteristic cohomology class $k \in \mathrm{Char}(G)$ satisfying $$e_j+2 \leq \langle k,v_j \rangle \leq -e_j$$ for each $j$. Construct a sequence of cohomology classes $k=k_0\sim k_1 \sim \ldots \sim k_n$, where $k_{j+1}$ is obtained from $k_j$ by choosing any vertex $v$  weighted by $e$ such that $$\langle k_j,v \rangle = -e,$$ and letting $k_{j+1}=k_j+2PD(v)$. This algorithm ends at the \emph{terminal} cohomology class $k_n= \ell$ where we have
\begin{equation}
\label{goodpath}
e_j \leq \langle \ell ,v_j \rangle \leq -e_j+2
\end{equation}
at each $j$. In this case, we form a sequence of cohomology classes $k=k_0\sim k_1 \sim \ldots \sim k_n = \ell$ called a \emph{full path}. We say that the initial cohomology class $k \in \mathrm{Char}(G)$ supports a \emph{good full path} if its terminal cohomology class $\ell \in \mathrm{Char}(G)$ satisfies \eqref{goodpath}.

\subsection{Laufer Sequences}

Let $\mathcal{J}$ be an index set. For a negative-definite AR plumbing graph $G$ with vertices $v_j$ for $j \in \mathcal{J}$, let $\mathbb{H}^+(G)$ denote the subset of the set functions $\phi \in \mathrm{Hom}(\mathrm{Char}(G), \mathcal{T}^+)$ satisfying the following property: for every $j \in \mathcal{J}$, $k \in \mathrm{Char}(G)$, and $m \in \mathbb{Z}^{>0}$, we require that

$$U^{m-\chi_k(v_j)}\phi(k+2PD(v_j))=U^m\phi(k) \text{ whenever } \chi_k(v_j) \leq0, \text{ and}$$
$$U^{m}\phi(k+2PD(v_j))=U^{m+\chi_k(v_j)}\phi(k) \text{ whenever } \chi_k(v_j)>0$$
see Section~\ref{AR} for the definition of $\chi_k(v_j)$ for $j \in \mathcal{J}$.

The module  $\mathbb{H}^+(G)$ automatically has an $\mathbb{F}[U]$-module structure. We can simply describe the dual of $\mathbb{H}^+(G)$.  Consider the set $\mathbb{Z}^{\geq 0}\times\mathrm{Char}(G)$ with a typical element $U^m\otimes k$. We form an equivalence relation $\sim$ on $\mathbb{Z}^{\geq 0}\times\mathrm{Char}(G)$ by the following rule:
$$U^{m+n}\otimes(k+2PD(v_j))\sim U^m\otimes k \text{ if } n\geq 0,$$
\noindent and 
$$ U^{m}\otimes(k+2PD(v_j))\sim U^{m-n}\otimes k \text{ if } n< 0.$$ 

Let $\mathbb{K}^+(G)$ denote the set of above equivalence classes. Let $(\mathbb{K}^+(G))^*$ denote its dual.  For any $l \in \mathbb{Z}^{>0}$, let $\mathrm{Ker}U^{l+1}$ denote the subgroup of  $\mathbb{H}^+(G)$ which is the kernel of the multiplication by $U^{l+1}$. We have the following isomorphism for every $l$ $$\mathrm{Ker}\,(U^{l+1})\to\mathrm{Hom} \left (\frac{ \mathbb{K}^+(G)}{\mathbb{Z}^{\geq l}\times{\mathrm{Char}(G)}},\mathbb{F} \right ), \ \ \phi(U^m\otimes k) \mapsto (U^m\phi(k))_0$$ where $(\cdot)_0$ denotes the projection to the degree $0$ subspace of  $\mathcal{T}^+$, see \cite[Lemma~2.3]{OS03a}.

Now we review the definition of the Laufer sequence by consulting \cite[Section~7]{Nem05} and \cite[Section 5.5]{KO17}. We can form a sequence $(k(i))_{i=0}^\infty$ in $\mathrm{Char}(G)$ recursively as follows: We begin with the canonical class $k_{\mathrm{can}}=k(0)$. Suppose $k(i)$ has already been constructed. Then we may find $k(i+1)$ by the following algorithm:
\begin{enumerate}
\item We construct a sequence $x_0,x_1,\dotsm x_l$. Let $x_0=k(i)+2\mathrm{PD}(v_0)$. Suppose $x_m$ has been found.  If there exists $j\in \mathcal{J} \setminus \{0\}$ such that $x_m(v_j)=-e_j$ then we let $x_{m+1}= x_m+2\mathrm{PD(v_j)}$. 
\item Otherwise, we stop and set $l=m$ and $k(i+1)= x_{l}$. 
\end{enumerate}  

The sequence $(k(i))_{i=0}^\infty$ is called the \emph{Laufer sequence} of the AR graph $G$. Since the elements of each sequence construct above satisfy $x_m\sim x_{m+1}$ for every $m=0,\dots,l-1$, the cohomology classes $k(i)$ satisfy the following relations in $\mathbb{K}^+(G)$:
\begin{align*}
U^{\chi_{k(i)}(v_0)}\otimes k(i) \sim  k(i+1) & \text{ if }\chi_{k(i)}(v_0)\geq 0,  \text{ and }\\
k(i)\sim U^{-\chi_{k(i)}(v_0)}\otimes k(i+1) & \text{ if }\chi_{k(i)}(v_0)< 0.
\end{align*}

Let $\tau(n)=\sum_{i=0}^{n-1} \chi_{k(i)}(v_0)$, with $\tau(0)=0$. Therefore, tau, delta and Laufer sequences are related naturally and interchangeably.

The elements of $\mathrm{Ker}U$ in $(\mathbb{K}^+(G)^*$ are also visible in the Laufer sequence. In a graded root $R$, a vertex is said to be a \emph{root vertex} if it has valency 1. The following lemma identifies root vertices of $R$ with elements of $\mathrm{Ker}U$.

\begin{lemma} [Lemma 5.4, \cite{KO17}]
\label{ktog}
Given $k\in  \mathbb{K}^+(G)$  such that $k^* \in \mathrm{Ker}U$, there exists a unique element $k(i_0)$ of the Laufer sequence such that $k \sim k(i_0)$.
\end{lemma}

The next lemma provides sufficient condition to describe the index of a characteristic class in the Laufer sequence.

\begin{lemma} [Lemma 5.5, \cite{KO17}] 
\label{laufdetect}
If $k \in \mathrm{Char}(G)$ satisfies 
$$e_j+2 \leq \langle k,v_j \rangle\leq -e_j-2, \text{ for all } j\in J,$$
\noindent then $k^* \in \mathrm{Ker}U$ and there exists a unique $i_0\in \mathbb{Z}^{\geq 1}$ such that $k=k(i_0)$. Further,  the index $i_0$ is the component of the cohomology class $\mathrm{PD}^{-1}(k(i_0)-k_{\mathrm{can}})/2$ on the central vertex.
\end{lemma}

We explore the existence of terminal cohomology classes in the Laufer sequence.

\begin{proposition}
\label{laufergood}
Any terminal cohomology class of a good full path appears in the Laufer sequence and corresponds to a leaf of the graded root $R$.
\end{proposition}

\begin{proof}
By N\'emethi's isomorphism described in \eqref{nemethi}, every element  $k \in \mathrm{Ker}\,U \subset \mathbb{K}^+(G)$ is equivalent to some cohomology class $k(i)$ in Laufer sequence. This element must satisfy $e_j \leq k(i)\leq -e_j$ for every $j$, otherwise it would be equivalent to some $U \otimes k'$ contradicting with the assumption that $k \in   \mathrm{Ker}\,U$. Hence $k(i)$ must be an element in a good full path. Iterating the algorithm described at the beginning of this subsection, we eventually arrive at a terminal cohomology class $k(i')$ for some $i'\geq i$. Therefore, it appears as a leaf of the graded root $R$, see \cite[Section~5]{KO17} for an exposition.
\end{proof}

\section{Homology Cobordism Group and Heegaard Floer Homology}
\label{Floer}

\subsection{Homology Cobordism Group}
\label{homcobordism}

Let $Y_0$ and $Y_1$ be two oriented homology spheres. The \emph{homology cobordism group} $\Theta_\Z^3$ is defined by forming a group structure on the set of homology spheres under the \emph{homology cobordism} $\sim$ given by the equivalence relation 
\[  Y_0 \sim Y_1 
 \iff  \begin{cases}
  \bullet \ \partial W = -(Y_0) \cup Y_1, \\
  \bullet \ H_*(W, \mathbb{Z}) = H_*(S^3 \times [0,1], \mathbb{Z}).  \end{cases} \]
for some compact, connected, oriented, smooth $4$-manifold $W$. 

We denote $[.]$ by element of $\Theta_\Z^3$. The connected sum induces the summation for $\Theta_\Z^3$, i.e., $[Y_0]+[Y_1] = [Y_0 \# Y_1]$. The zero element is provided by $[S^3]$, and additive inverses in $\Theta_\Z^3$ are obtained by taking orientation reversal, i.e. $-[Y]=[-Y]$.

\subsection{Involutive Heegaard Floer Homology}
\label{involutive}

In \cite{HMZ18}, Hendricks, Manolescu, and Zemke provides a homomorphism that sends the homology cobordism class of a homology sphere to the local equivalence class of its $\iota$-complex up to a graded shift 
\begin{equation}
\label{hm}
h: \Theta_\mathbb{Z}^3 \to \mathfrak{I}, \ \ h([Y]) = (CF^-(Y), \iota)[-2].
\end{equation}

Since we work with Brieskorn spheres, their Heegaard Floer homologies carry the same information of graded roots, so we can analyze the local equivalence classes of Brieskorn spheres by deciphering their graded roots and monotone subroots. This approach relies on the following theorem of Dai and Manolescu.

\begin{theorem}[Theorem 6.1, \cite{DM17}; Theorem 4.2, \cite{DS17} ]
Every graded root is uniquely locally equivalent to its monotone subroot up to ordered set of parameters. Further, any monotone root $M=M(h_1,r_1; \ldots; h_n,r_n)$, we have the local equivalence $$M = \left ( \sum_{i=1}^{n} M(h_i,r_i) \right ) - \left ( \sum_{i=1}^{n-1} M(h_{i+1},r_i) \right ) .$$
\end{theorem}

One can parametrize the monotone subroot of Brieskorn spheres by using their $d$-invariants and Neumann-Siebenmann invariants $\bar{\mu}$. Thanks to the result of Dai, one can also read off $\bar{\mu}$ from the graded roots \cite{D18}. More precisely, we use the following recipe.

\begin{theorem}[Section 8, \cite{DM17}]
\label{parametrization}
Let $Y$ be a AR plumbed homology sphere. If the monotone root $M(h_1,r_1; \ldots; h_n,r_n)$ corresponds to the local equivalence of $h([Y])$ in the sense of \eqref{hm}, then we set $h_i = d_i(Y)-2$ and $r_i = -2\bar{\mu}_i(Y) -2$ for $i=1, \ldots, n$ where $d_1(Y) = d(Y) $ and $\bar{\mu}_n (Y) = \bar{\mu}(Y)$.
\end{theorem}

Modifying the group $\mathfrak{I}$, in \cite{DHST18}, Dai, Hom, Stoffregen and Truong studied the notion of almost $\iota$-complexes by forming another group $\widehat{\mathfrak{I}}$. There is a close relation between the new group $\widehat{\mathfrak{I}}$ and the homology cobordism group $\Theta_\mathbb{Z}^3$ by a homomorphism which factors through $\mathfrak{I}$. This homomorphism is given by $\widehat{h}: \Theta_\mathbb{Z}^3 \to \mathfrak{I} \to \widehat{\mathfrak{I}}$ where $\mathfrak{I} \to \widehat{\mathfrak{I}}$ is a a forgetful map. Unlike the group $\mathfrak{I}$, it is possible to classify the group structure of $\widehat{\mathfrak{I}}$ completely.

\begin{theorem}[Theorem 6.2, \cite{DHST18}]
Every almost $\iota$-complex is locally equivalent to a standard complex of the form $\mathcal{C} = \mathcal{C}(a_1,b_2, \ldots, a_{2n-1}, b_{2n})$ where $a_i \in \{ +,- \}$ and $b_i \in \mathbb{Z} \setminus \{ 0\}$.
\end{theorem}

This classification enables to define homomorphisms $\phi_n: \widehat{\mathfrak{I}} \to \mathbb{Z}$ on each almost local equivalence class by the rule 
\begin{equation}
\label{phi}
\phi_n(\mathcal{C}) = \# \{ b_i = n\} - \# \{ b_i = -n\}.
\end{equation}

Combining the maps $\phi_n$'s with the homomorphism $\widehat{h}$, we have the desired homomorphisms $f_n \doteq \phi_n \circ \widehat{h}$.

\begin{theorem}[Theorem~7.18, Theorem~7.19, \cite{DHST18}]
\label{count}
There are countably infinitely many group homomorphisms $$\{f_n\}_{ n \in \mathbb{N} } : \Theta_\mathbb{Z}^3 \to \mathbb{Z}^\infty.$$
\end{theorem}

\subsection{Connected Heegaard Floer Homology}
\label{connected}

Let $Y$ be a homology sphere. Studying the maximal self-local equivalence classes of $\iota$-complexes up to homotopy, in \cite{HHL18} Hendricks, Hom and Lidman defined a homology cobordism invariant $HF_{\mathrm{conn}}(Y)$, called \emph{connected Heegaard Floer homology}. 

They proved that $HF_{\mathrm{conn}}(Y)$ is isomorphic to a summand of ordinary reduced Heegaard Floer homology $HF_{\mathrm{red}}(Y)$, see [Theorem 1.1, \cite{HHL18}]. They also present the computation procedure of the connected Heegaard Floer homology for homology spheres whose minus version of Heegaard Floer homology $HF^-(Y)$ is represented by the $\mathbb{F}[U]$-module associated to graded root $R$ up to a grading shift:

\begin{theorem}[Theorem~1.16, \cite{HHL18}]
\label{connectedHF}
Let $Y$ be an AR plumbed homology sphere such that  $HF^-(Y) \simeq \mathbb{H}^-(R)$ for some graded root $R$. Then $HF_{conn}(Y)$ is the $U$-torsion submodule of $\mathbb{H}^-(M)$ for some monotone graded subroot $M$ of $R$, shifted upward in degree by $1$.
\end{theorem}

\begin{corollary}[Corollary~1.17, \cite{HHL18}]
\label{connectedHF2}
Let $Y$ be an AR plumbed homology sphere such that $\mathrm{HF}^-(Y) \cong \mathbb{H}^-(M)$ for some graded root $R$ with
$\mathrm{HF}_{red} (Y) \cong \bigoplus_{i=1}^N \big(\mathcal{T}_{a_i}^+ (n_i) \big)^{k_i}$, where $(a_i, n_i) \neq (a_j, n_j)$ if $i \neq j$.
Then
\[ HF_{\mathrm{conn}}(Y) \cong \bigoplus_{i=1}^N \big(\mathcal{T}_{a_i}^+ (n_i) \big)^{k'_i}, \ \ \ \text{where} \ \ k'_i =
\begin{cases}
0 & \text{ if } k_i \text{ even,} \\
1 & \text{ if } k_i \text{ odd.}
\end{cases}
\]
\end{corollary}

\section{Packed Sequences, Monotone Roots and Connected Heegaard Floer Homology}
\label{packed}

First, we will define auxiliary sequences that help us to understand semigroups of Brieskorn spheres hence their delta sequences.

\subsection{Packed Sequences and Their Complementary Sequences}

Consider the Brieskorn sphere $\Sigma(p,q,r)$ that can be realized as the boundaries of ASL graphs. Set $x = pq$, $y=pr$ and $z=qr$. Then for $\Sigma(p,q,r)$, we have $N_0=pqr-x-y-z$. Note a worthy identity that 
\begin{equation}
\label{xyz}
x+y=z+1
\end{equation}

Clearly, any element of $S_{Y}$ can be written as $fx+gy+hz$ for some non-negative integers $f$,$g$ and $h$. This decomposition is unique by the Chinese remainder theorem. 

\begin{definition}
Let $f$, $g$, and $h$ be non-negative integers such that $\mathrm{min}\{f,g\}=0$. We say that $\mathbf{[fx+gy+hz]} $ is a packed sequence if it is an integer-valued sequence of elements of $S_Y$ of the following form

\vspace{0.2cm}
\begin{center}
$\mathbf{[fx+gy+hz]} =  fx+gy+hz, (f+1)x+(g+1)y+(h-1)z, \ldots, (f+h)x+(g+h)y .$
\end{center}
\vspace{0.2cm}
\end{definition}

\noindent In the short-hand notation of packed sequences, we apply the following rule: each time we subtract one from the coefficient of $z$ and then we add one to coefficients of both $x$ and $y$. We call $fx+gy+hz$ and $(f+h)x+(g+h)y$, the \emph{initial} and respectively the \emph{final} elements of the packed sequence $\mathbf{[fx+gy+hz]}$.

Recall that $N_0 = pqr - x - y -z$. In a similar vein, we define the notion of complementary packed sequences.

\begin{definition}
Let $f$, $g$, and $h$ be non-negative integers such that $\mathrm{min}\{f,g\}=0$. The integer-valued sequence of elements of $Q_Y$ of the following form is called a complementary packed sequence
\vspace{0.2cm}
\begin{center}
$N_0 - \mathbf{[fx + gy + hz]} =N_0-(f+h)x-(g+h)y,\dots,N_0-fx-gy-hz.$
\end{center}
\end{definition}

\noindent The same rule holds for complementary packed sequences. Respectively, $N_0-(f+h)x-(g+h)y$ and $N_0-fx-gy-hz$ are called the \emph{initial} and \emph{final} elements of the complementary packed sequence $N_0 - \mathbf{[fx+gy+hz]}$.

\begin{lemma}
\label{consecutive}
The packed sequence $[\mathbf{fx+gy+hz}] $ and the complementary packed sequence $N_0 - \mathbf{[fx + gy + hz]}$ both contain exactly $h+1$ consecutive integers each of which is contained in the semigroup $S_{Y_n}$ and $Q_{Y}$ respectively.
\end{lemma}

\begin{proof}
This claim is an easy consequence of the identity \eqref{xyz}.
\end{proof}

Thanks to the above lemma we can regard packed and complementary sequences as inseparable blocks in $S_{Y_n}$ and $Q_{Y_n}$ respectively. Since the elements of these sequences are consecutive integers, they are also consecutive elements respectively in  $S_{Y_n}$ and $Q_{Y_n}$ as well. Hence no other element of the semigroups can appear between the elements of these sequences. 

Let us define the following order rules for packed sequences. For a non-negative integer $w$, we write $[\mathbf{fx+gy+hz}] \leq w$ if every element in  $[\mathbf{fx+gy+hz}] $ is less than or equal to $w$. We indicate $[\mathbf{f_1x+g_1y+h_1z}]< [\mathbf{f_2x+g_2y+h_2z}]$ if every element in $[\mathbf{f_1x+g_1y+h_1z}]$ is strictly less than every element in  $[\mathbf{f_2x+g_2y+h_2z}]$. Clearly $[\mathbf{f_1x+g_1y+h_1z}]< [\mathbf{f_2x+g_2y+h_2z}]$ holds if and only if the initial elements satisfy $f_1x+g_1y+h_1z < f_2x+g_2y+h_2z$ which is the case if and only if the final elements satisfy $(f_1+h_1)x+(g_1+h_1)y< (f_2+h_2)x+(g_2+h_2)y$. 

Since complementary packed sequences are defined by subtracting packed sequences from the integer $N_0$, we can extend the above discussion naturally to complementary packed sequences by changing the directions of inequalities oppositely.

Now, our aim is to prove our main theorem together with the re-computation of connected Heegaard Floer homology of the family of Dai, Hom, Stoffregen and Truong.

\subsection{Computations for $ \{ X_n=\Sigma(2n+1,4n+1, 4n+3) \}^\infty_{n=1}$ }
\label{Xn}

By \cite[Lemma~4.6, Theorem~1.2]{KS20}, we know that the terminal cohomology classes $k_1$ and $k_2$ of good full paths for $X_n$ are the ones shown in Figure~\ref{fig:max} and $d(X_n) = -2n$ for $n \geq 1$. Here, we abuse of notation by indicating each vertex $v_i$ as the number $\langle k, v_i \rangle$.

\begin{figure}[htbp]
\begin{center}
\includegraphics[width=0.80\columnwidth]{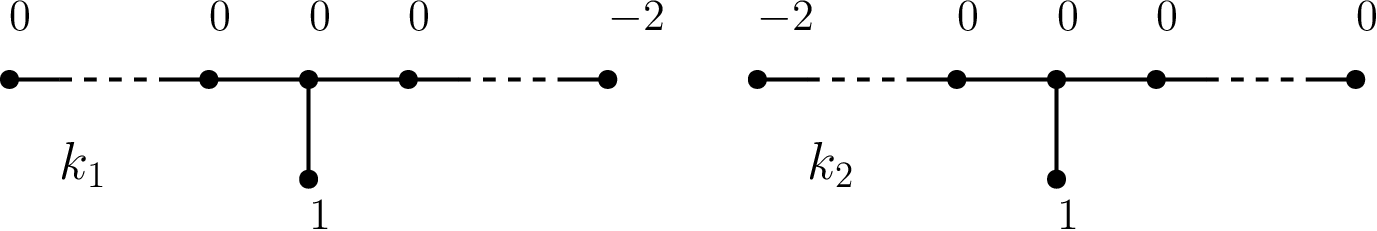}
\end{center}
\caption{The terminal cohomology classes $k_1$ and $k_2$.}
\label{fig:max}
\end{figure}

\begin{lemma}
\label{ab}
For $\{ X_n \}^\infty_{n=1}$, the positions of $k_1$ and $k_2$ in the Laufer sequence are
\begin{align*}
a=x+(n-1)z \ \ \text{and} \ \ b=y+(n-1)z. 
\end{align*}
Thus $a,b \in S_{X_n}.$
\end{lemma}

\begin{proof}
By Lemma \ref{ktog} and Lemma \ref{laufergood}, we know that $k_1$ and $k_2$ appear in the Laufer sequence. Further, their positions can be found by using Lemma \ref{laufdetect} as follows.

The canonical class of $\{X_n\}^\infty_{n=1}$ is $k_{\mathrm{can}} = (2n-1, 0, \ldots, 0)$. The position of $k_1$ in the Laufer sequence $a$ is the component of cohomology class $\mathrm{PD}^{-1}(k_{\mathrm{can}}-k_1)/2$ on the central vertex $v_q$. To compute $b$, we change the role of $k_1$ with $k_2$. It turns out that 

\begin{align*}
a &= \left [ I^{-1} . \left ( \frac{k_{\mathrm{can}} - k_1}{2} \right ) \right ]_{qq}  \\
&= I^{-1}_{v_qv_0} . (n-1) + I^{-1}_{v_qv_{q+r-1}}. 1  \\
&= (4n+1)(4n+3)(n-1) + (2n+1)(4n+1)   \\
&= x+(n-1)z
\end{align*}
and
\begin{align*}
b =& \left [ I^{-1} . \left ( \frac{k_{\mathrm{can}} - k_2}{2} \right ) \right ]_{qq}  \\
&= I^{-1}_{v_qv_0} . (n-1) + I^{-1}_{v_qv_1}.1  \\
&= (4n+1)(4n+3)(n-1) + (2n+1)(4n+1)   \\
&= y+(n-1)z
\end{align*}

\noindent where $[.]_{qq}$ denotes $(q,q)$-entry of the column matrix. Here, we find the entries of the inverse intersection matrix $I^{-1}$ by using the recipe in Section~\ref{aslgraphs}.
\end{proof}

For $\{X_n\}^\infty_{n=1}$, we consider only one packed sequence $\mathbf{x+(n-1)z}$ with the initial element $a$. Clearly, $a < b$. We easily and completely characterize the elements lying in $S_{X_n} \cap [a,b]$ and $X_{X_n} \cap [a,b]$, compare with \cite{HKL16} and \cite{St17}.

\begin{lemma}
\label{SXn}
Any  element in $S_{X_n} \cap [a,b)$ must belong to the sequence $\mathbf{x+(n-1)z}$. Thus, $S_{X_n} \cap [a,b]$ is given by the sequence $\mathbf{x+(n-1)z}$ and the element $b$ as an ordered set.
\end{lemma}

\begin{proof}
Assume for a contradiction that there exists another element $d \in S_{X_n} \cap [a,b)$. Without loss of generality suppose that $d$ is an initial element of a packed sequence. So we have either $d=gy +hz$ or $d=fx+hz$ where  $0 \leq h \leq n-1$.

Consider the first case and extend the first equality to the last element of $\mathbf{x+(n-1)z}$. Now we have the strict inequalities
\begin{align*}
nx+(n-1)y < gy+hz<y+(n-1)z.
\end{align*} 
Applying the identity \eqref{xyz}, we obtain 
\begin{align*}
(h-n)x + (g+h+1-n)y -h > 0 \ \ \text{and} \ \ (n-h-1)x + (n-g-h)y +h >0
\end{align*}
Clearly, both these inequalities are no longer to be true at the same time. The second case can be eliminated similarly. We have
\begin{align*}
nx+(n-1)y < fx+hz<y+(n-1)z.
\end{align*} 
By using the identity \eqref{xyz}, we obtain
\begin{align*}
(f+h-n)x + (h+1-n)y -h > 0 \ \ \text{and} \ \ (n-h-f-1)x + (n-h)y +h >0
\end{align*}
Again, this would not be the case.
\end{proof}

\begin{lemma}
\label{XXn}
As an ordered set, the elements of $X_{X_n} \cap [a,b]$ are respectively given by the following sequences and elements. We have
\begin{align*}
\mathbf{x+(n-1)z}, \ N_0 - (\mathbf{x+(n-1)z}), \ b
\end{align*} 
\end{lemma}

\begin{proof}
To verify both inequalities, it suffices to compare initial or final elements of packed sequences. We first justify the first inequality. Note that $N_0 = 2nz - x- y$ since $p=2n+1$. Then by using the identity \eqref{xyz}, we can write
\begin{align*}
N_0 - 2x - 2(n-1)z = 2z - 3x -y = y-x-2 = 4n> 0
\end{align*}
since $x=(2n+1)(4n+1)$ and $y=(2n+1)(4n+3)$. For the second equality, again we plug in the identity \eqref{xyz} and apply $N_0 = 2nz - x- y$ as follows:
\begin{align*}
x+y + 2(n-1)z -N_0 = 2x+2y-2z = 2 > 0.
\end{align*}
\end{proof}

\begin{proposition}
\label{monotonesubroot3}
For $\{ X_n \}^\infty_{n=1}$, the reduced delta sequence $\tilde{\Delta}_{X_n}$ on $[a,b]$ is given by
\begin{center}
$ \left \langle n, -n,1 \right \rangle.$
\end{center}
\noindent Therefore, the corresponding minimal graded subroot $R_{X_n}$ and the monotone subroot $M_{X_n}$ of $X_n$ are shown in Figure~\ref{fig:rootsofXn}.
\end{proposition}

\begin{proof}
By Lemma \ref{XXn}, $\{ X_n \}^\infty_{n=1}$ has the wanted reduced delta sequence, see Section~\ref{delta}. Then the corresponding minimal graded root $R_{X_n}$ is constructed by using the recipe in Section~\ref{graded}. In this case, the monotone subroot $M_{X_n}$ is same with $R_{X_n}$ due to Section~\ref{monotoneroot} and Theorem~\ref{parametrization}.
\end{proof}

\begin{figure}[htbp]
\begin{center}
\includegraphics[width=0.5\columnwidth]{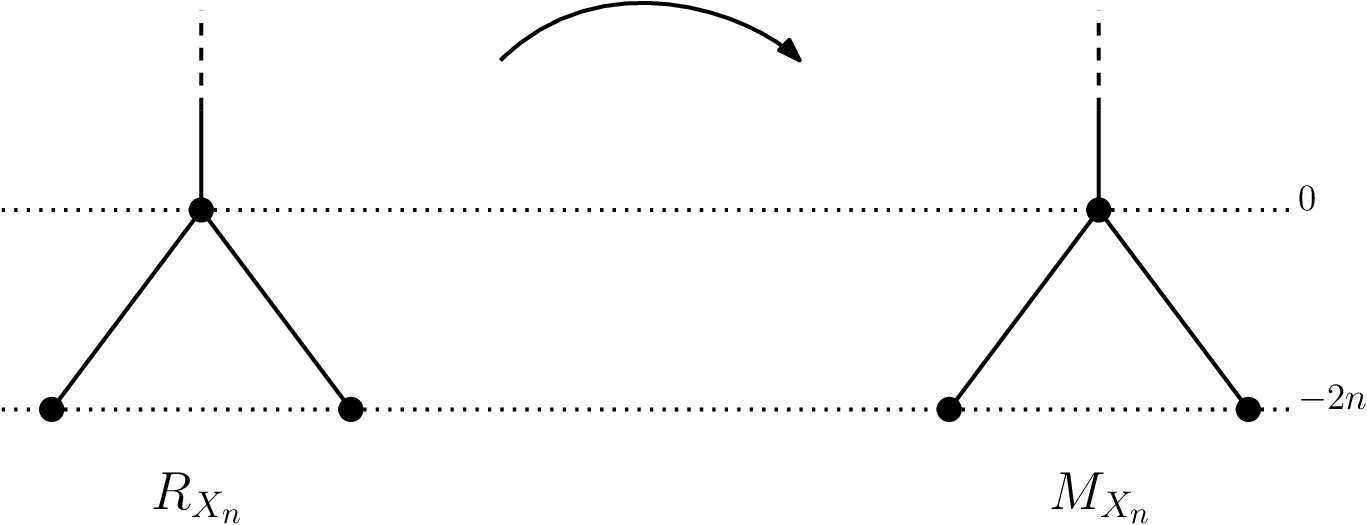}
\end{center}
\caption{The minimal graded and monotone subroots of $\{ X_n \}^\infty_{n=1}$.}
\label{fig:rootsofXn}
\end{figure}

\subsection{Computations for $\{ Y_n=\Sigma(2n+1,3n+2, 6n+1) \}^\infty_{n=1}$}

Since the lattice point realizing the $d$-invariant of $\{Y_n\}^\infty_{n=1}$ is same with one of $\{X_n\}^\infty_{n=1}$, by \cite[Lemma~4.6, Theorem~1.2]{KS20}, the values of $a$ and $b$ has the same formula but the values of $x,y$ and $z$ are different. Note that $d(Y_n)=-2n$ for $n \geq 1$.

\begin{lemma}
\label{ab2}
For $\{Y_n\}^\infty_{n=1}$, the positions of terminal cohomology classes $k_1$ and $k_2$ of good full paths for $\{Y_n\}^\infty_{n=1}$ in the Laufer sequence are
\begin{align*}
a=x+(n-1)z \ \ \text{and} \ \ b=y+(n-1)z. 
\end{align*}
Thus $a,b \in S_{Y_n}.$
\end{lemma}

\begin{proof}
Notice that $x=(2n+1)(3n+2)$, $y=(2n+1)(6n+1)$ and $z=(3n+2)(6n+1)$ for $\{Y_n\}^\infty_{n=1}$. The proof of lemma is identical with Lemma~\ref{ab} with exception of values of $x,y$ and $z$. 
\end{proof}

Next we define two special packed sequences for $\{Y_n\}^\infty_{n=1}$:
\begin{align*}
\mathbf{\theta_k}&= \mathbf{[(1+3k)x + (n-1-k)z]} \text{ for all } k \in \{ 0,1, \ldots, n-1 \}, \\
\mathbf{\omega_\ell}&=\mathbf{[(1+3\ell)y + (n-1-2\ell)z]} \text{ for all } \ell \in \left \{ 1,2, \ldots,\left \lfloor \frac{n-1}{2} \right \rfloor \right \}.
\end{align*}

By Lemma~\ref{consecutive}, it is clear that $\mathbf{\theta_k}$  contains $n-k$ elements and  $\mathbf{\omega_\ell}$ contains $n-2\ell$ elements.

\begin{lemma}
\label{characterization} 
For the packed sequences $\mathbf{\theta_k}$ and $\mathbf{\omega_\ell}$, the following facts hold:
	\begin{enumerate}
			\item\label{part:theta}$\mathbf{\theta_{k_1}}<\mathbf{\theta_{k_2}}$  for all $k_1,k_2\in \{0,1,\dots,n-1\}$ with $k_1<k_2$.
		\item \label{part:omega} $\mathbf{\omega_{\ell_2}}<\mathbf{\omega_{\ell_1}}$ for all $\ell_1,\ell_2\in  \left \{ 1,2, \dots,\left \lfloor \frac{n-1}{2} \right  \rfloor \right \}$ with $\ell_1<\ell_2$.
		\item  $a\leq \mathbf{\theta_k} < b$ for all $k\in \{0,1,\dots,n-1\}$. 
		\item $a< \mathbf{\omega_\ell} < b$ for all  $ \ell \in \left \{ 1,2, \dots,\left \lfloor \frac{n-1}{2} \right  \rfloor \right \}$.
		\item $\mathbf{\theta_{n-\ell-1}}< \mathbf{\omega_\ell}<\mathbf{\theta_{n-\ell}}$ for all $ \ell \in \left \{ 1,2, \dots, \left \lfloor \frac{n-1}{2} \right  \rfloor \right \}$.
		\end{enumerate}
\end{lemma}

\begin{proof}
We consider the proof case by case.
\begin{enumerate}
\item There is nothing to prove when $n=1$. Assume $n\geq 2$. It suffices to show that $\mathbf{\theta_{k}}<\mathbf{\theta_{k+1}}$ for all $k\in \{0,1, \dots,n-2\}$. To check the latter it is enough to compare the initial elements of the respective packed sequences, so the claim is equivalent to showing $ (1+3k)x+(n-1-k)z < (1+3(k+1))x+(n-1-k)z$, which is true because $3x-z=6n+4>0$ for all $n$. 
		
			\vspace{0.2cm}
		
		\item The claim is trivial when $n\leq 4$, so assume $n\geq 5$. It suffices to show that $\mathbf{\omega_{\ell+1}} < \mathbf{\omega_{\ell}}$ for all $\ell \in \left \{ 1,2, \dots,\lfloor \frac{n-1}{2} \rfloor -1 \right \}$. To prove the latter we again compare the initial elements and show $(1+3(\ell+1))y+(n-1-2(\ell+1))z<(1+3\ell)y+(n-2-2\ell)z$. This inequality is equivalent to $3y-2z<0$ which is true because $3y-2z=-(6n+1)$. 
		
		\vspace{0.2cm}
		
		\item Notice that $a$ is the initial element of $\theta_0$, hence by part \eqref{part:theta} we have $a\leq \mathbf{\theta_k}$ for all $k$. To prove the rest, observe that  $\mathbf{\theta_{n-1}}$ consists of a single element namely $(1+3(n-1))x$, so it suffices to show this element is stricly less than $b$. The required inequality is  $(1+3(n-1))x< y+(n-1)z$, which is equivalent to $n(y-2x-1) + x + 1>0$ by the identity \eqref{xyz}. The latter is true for all $n$ because $y-2x = -(6n+3)$ and $x=6n^2 + 7n +2$.

		\vspace{0.2cm}

		\item By part \eqref{part:omega}, we need to see that the last element of $\mathbf{\omega_1}$ is less than $b$, i.e., $(n-3)x + (n+1)y < y + (n-1)z$. Using the identity \eqref{xyz}, it suffices to show that $2x-y + 1-n > 0$. It is true for all $n$ because $2x-y=6n+3$. For the rest it is enough to verify that $a$ is less than the initial element of $\mathbf{\omega_{ \left \lfloor \frac{n-1}{2} \right  \rfloor}}$, i.e., we need to have $x + (n-1)z < \left (1+ 3 \left \lfloor \frac{n-1}{2}  \right \rfloor \right )y+\left (n-1 - 2 \left \lfloor \frac{n-1}{2} \right \rfloor \right )z$. It is equivalent to showing that $y-x + \left \lfloor \frac{n-1}{2} \right \rfloor (3y-2z) > 0$ for all $n$. This is true since $y-x = (2n+1)(3n-1)$ and $3y-2x = -(6n+1)$.
		
		\vspace{0.2cm}
		
		\item We compare the initial elements of relevant packed sequences. We show that the inequalities $(1+3(n- \ell -1))x + \ell z < (1+3 \ell )y + (n-1-2\ell )z$ and $(1+3\ell )y + (n-1-2\ell )z < (1+3(n-\ell ))x + (\ell -1)z$ both hold. Using the identity \eqref{xyz}, we reduce the former inequality $n(y-2x-1) + x + 3\ell +1 > 0$. Since $y-2x = -(6n+3)$ and $x = 6n^2 +7n +2$, this inequality is true for all $n$. Again by the identity \eqref{xyz}, we transform the latter inequality to $n(2x-y+1) + x - y - 3\ell >0$ which is true for all $n$ since $2x-y = 6n+3$ and $x-y = -6n^2-n+1$. 
	\end{enumerate}
\end{proof}

\begin{lemma}
\label{characterization2}
Any  element in $S_{Y_n} \cap [a,b)$ must belong to one of the packed sequences $\mathbf{\theta_k}$ for some $ k \in \{ 0,1, \ldots, n-1 \}$ or $\mathbf{\omega_\ell}$ for some $ \ell \in \left \{ 1,2, \ldots, \lfloor \frac{n-1}{2} \rfloor \right \}$.
\end{lemma}

\begin{proof}

Assume for a contradiction that there exists another element $d \in S_{Y_n} \cap [a,b)$. Without loss of generality suppose that $d$ is an initial element of a packed sequence. We consider the following two cases.

\textbf{Case I.} $d=fx+hz$ where  $0 \leq h \leq n-1$.

We can re-parametrize $d$ as $d=fx+(n-1-k)z$ for  $k\in\{0,1,\dots,n-1\}$. We will prove that $d$ must be the initial element of $\mathbf{\theta_k}$. By Lemma~\ref{characterization}, we have $a \leq d <b$, so
\begin{align*}
x+(n-1)z\leq fx+(n-1-k)z<y+(n-1)z,
\end{align*} 
\noindent which implies
\begin{align*}
0\leq (f-1)x-kz<y-x.
\end{align*}
\noindent Dividing by $x$, we get
\begin{align*}
0\leq f-1-k\frac{z}{x}<\frac{y}{x}-1.
\end{align*}
Notice that $\frac{y}{x}<2 $ for all $n$, hence
\begin{align*}
0\leq f-1-k\frac{z}{x}<1.
\end{align*}

\noindent We write $\frac{z}{x}=3-\frac{2}{2n+1}$ and re-arrange the terms as follows
\begin{align*}
-\frac{2k}{2n+1}\leq f-1-3k<1-\frac{2k}{2n+1}.
\end{align*}
 Now as $0\leq \frac{2k}{2n+1}<1$, we have
 \begin{align*}
 -1<f-1+3k<1
 \end{align*}
 The only integer $f$ satisfying the above inequality is $f=1+3k$. 
 
\textbf{Case II.} $d=gy +hz$ where $0 \leq h \leq n-1$.

We will follow a similar strategy as in Case I and show that $d$ is the initial element of the packed sequence $\mathbf{\omega_\ell}$ for some $\ell$. Write $h=n-1-k$ for $k\in \{0,1,\dots,n-1\}$.  Since $a$ is the initial element of $\mathbf{\theta_0}$ and $d$ is not of this form, by Lemma~\ref{characterization}, we have the strict inequality $a<d<b$, so
\begin{align*}
x+(n-1)z<gy+(n-1-k)z<y+(n-1)z.
\end{align*}
Subtract $y+(n-1)z$ to get
\begin{align*}
x-y<(g-1)y-kz<0.
\end{align*}
Dividing by $y$ we obtain
\begin{align*}
\frac{x}{y}-1<g-1-k\frac{z}{y}<0.
\end{align*}
Plug in $\frac{x}{y}=\frac{3n+2}{6n+1}=\frac{1}{2}+\frac{3}{12n+2}$ and $\frac{z}{y}=\frac{3n+2}{2n+1}=\frac{3}{2}+\frac{1}{4n+2}$ and re-arranging we have
\begin{align*}
-\frac{1}{2}+\frac{3}{12n+2}+\frac{k}{4n+2}<g-1-\frac{3k}{2}<\frac{k}{4n+2}
\end{align*}
Since $0\leq k \leq n-1$, we find
\begin{align*}
-\frac{1}{4}<g-1-\frac{3k}{2}<\frac{1}{4}
\end{align*}
We see that when $k$ is odd the above inequality has not an integer solution for $g$. When $k$ is even, we can write $k=2\ell$ for $\ell \in \left \{1,2,\dots,\lfloor\frac{n-1}{2}\rfloor \right \}$ and $g=1+3 \ell$ as required.
\end{proof}

Using preceding lemmas, we precisely describe all elements lying in $S_{Y_n} \cap [a,b]$.

\begin{lemma}
\label{SYn}
As an ordered set, $S_{Y_n} \cap [a,b]$ is given by the following sequences and elements \begin{center}
$\mathbf{\theta_0}, 
\mathbf{\theta_1}, 
\dots,  
\mathbf{\theta_{\left \lfloor \frac{n}{2} \right  \rfloor }}, \mathbf{\omega_{\left \lfloor \frac{n-1}{2} \right \rfloor}}, \mathbf{\theta_{\left \lfloor \frac{n}{2}  \right  \rfloor +1 }}, \mathbf{\omega_{\left \lfloor \frac{n-1}{2} \right \rfloor -1}},
\dots, 
\mathbf{\theta_{n-2}}, 
\mathbf{\omega_1}, 
\mathbf{\theta_{n-1}},
b$
\end{center}

\end{lemma}

\begin{proof}
The proof follows the combination of Lemma~\ref{ab}, Lemma~\ref{consecutive}, Lemma~\ref{characterization}, and Lemma~\ref{characterization2}.
\end{proof}

\begin{lemma}
\label{complementary} 
The complementary packed sequences satisfy the following inequalities:
	\begin{enumerate}
		\item $\mathbf{\theta_{k-1}}<N_0-\mathbf{\theta_{n-k}}<\mathbf{\theta_k}$ for all $k\in \{1,2, \dots,n-1 \}$,
		\item $\mathbf{\theta_{\ell -1}}<N_0-\mathbf{\omega_{\ell }}<\mathbf{\theta_\ell }$ for all $\ell \in \left \{1,2,\dots,\lfloor \frac{n-1}{2} \rfloor \right \},$
	\end{enumerate}
	
\end{lemma}

\begin{proof}
Since $pqr=pz$ and $p=2n+1$, using the identity \eqref{xyz}, we write
\begin{align}
\label{N0}
N_0 = (2n-1)x + (2n-1)y -2n
\end{align}

\begin{enumerate}
\item We compare the initial elements of relevant packed sequences. We show that the inequalities $(1+3(k-1))x + (n-k) z < N_0 - (1+3(n-k) )x - (k-1)z$ and $N_0 -  (1+3(n-k) )x - (k-1)z < (1+3k)x + (n-k-1) z$ both hold. The former inequality is a consequence following equalities:
\begin{align*}
& N_0 - \left ( (1+3(n-k) )x + (k-1)z + 1+3(k-1))x + (n-k) z \right ) \\
& = N_0 - \left ( (3n-1)x + (n-1)z \right ) \\
& {\overset{\eqref{xyz}} =} N_0 - \left( (4n-2)x + (n-1)y -n+1  \right) \\
& {\overset{\eqref{N0}} =} (2n-1)x + (2n-1)y -2n - \left( (4n-2)x + (n-1)y -n+1  \right) \\
& = n(y-2x) + x -n -1 > 0
\end{align*}
 
\noindent This is true since $y-2x = -(6n+3)$ and $x = 6n^2 + 7n +2$. We can derive the second inequality for all $n$ since $2x-y = 6n+1$:
\begin{align*}
& (1+3(n-k) )x + (k-1)z + (1+3k)x + (n-k-1) z - N_0 \\
& = (3n+2 )x + (n-2)z - N_0 \\
& {\overset{\eqref{xyz}} =} 4nx + ny -2y + n +2 - N_0 \\
& {\overset{\eqref{N0}} =} 4nx + ny -2y + n +2 - \left( (2n-1)x + (2n-1)y -2n \right) \\
& = n(2x-y) + x +y +2 >0
\end{align*}

			\vspace{0.2cm}
\item In this case, we follow the same strategy by comparing the initial elements of our sequences. It suffices to prove that the inequalities $(1 + 3(\ell-1) )x + (n- \ell )z < N_0 - (1+3\ell )y - (n-1-2\ell )z$ and $N_0 - (1+3\ell )y - (n-1-2\ell )z < (1+3\ell )x + (n-1-\ell )z$ are true for all $n$. To see the first equality, we do the following calculations:
\begin{align*}
& N_0 - \left( (1+3\ell )y + (n-1-2\ell )z + 3(\ell-1) )x + (n- \ell )z \right) \\
& = N_0 - \left( 3(\ell-1) )x + (1+3\ell )y + (2n-3\ell -1)z  \right) \\
& {\overset{\eqref{xyz}} =} N_0 - \left( (2n-3)x + 2ny -2n +3\ell +1 \right) \\
& {\overset{\eqref{N0}} =} (2n-1)x + (2n-1)y -2n - \left( (2n-3)x + 2ny -2n +3\ell +1 \right) \\
& = 2x-y -3\ell -1 > 0
\end{align*}

\noindent The latter one is true since $2x -y = 6n+3$. The last inequality also holds because:
\begin{align*}
& (1+3\ell )x + (n-1-\ell )z + (1+3\ell )y + (n-1-2\ell )z - N_0 \\
& = (1+3\ell )x + (1+3\ell )y + (2n-2- 3\ell )z - N_0 \\
& {\overset{\eqref{xyz}} =} (2n-1)x + (2n-1)y -2n+2 +3\ell - N_0 \\
& {\overset{\eqref{N0}} =} (2n-1)x + (2n-1)y -2n+2 +3\ell - \left( (2n-1)x + (2n-1)y -2n \right) \\
&= 3\ell +2 >0
\end{align*}

\end{enumerate}
\end{proof}

Now we are able to list all elements inside $X_{Y_n} \cap [a,b]$.

\begin{lemma}
\label{XYn}
As an ordered set, the elements of $X_{Y_n} \cap [a,b]$ are respectively given by the following sequences and elements. When $n$ is odd, we have
\begin{center}
$\mathbf{\theta_0}, 
N_0 - \mathbf{\theta_{n-1}}, N_0 - \mathbf{\omega_1},
\dots,
\mathbf{\theta_{ \frac{n-3}{2} }},
N_0 - \mathbf{\theta_{\frac{n+1}{2}  }},
N_0 - \mathbf{\omega_{ \frac{n-1}{2} }}, 
\mathbf{\theta_{ \frac{n-1}{2}  }}, 
N_0 - \mathbf{\theta_{ \frac{n-1}{2} }},
\newline  
\mathbf{\omega_{ \frac{n-1}{2} }},
\mathbf{\theta_{ \frac{n+1}{2}  }}, 
N_0 - \mathbf{\theta_{ \frac{n-3}{2} }}, 
\dots,  
\mathbf{\omega_1},  
\mathbf{\theta_{n-1}},
N_0 - \mathbf{\theta_0}, 
b$
\end{center}
\noindent For even values of $n$, we also have
\begin{center}
$\mathbf{\theta_0},
N_0 - \mathbf{\theta_{n-1}}, 
N_0 - \mathbf{\omega_1},
\dots,
\mathbf{\theta_{ \frac{n-4}{2} }}, 
N_0 - \mathbf{\theta_{ \frac{n+2}{2} }},  
N_0 - \mathbf{\omega_{\frac{n-2}{2} }}, 
\mathbf{\theta_{\frac{n-2}{2} }}, 
N_0 - \mathbf{\theta_{ \frac{n}{2} }},
\mathbf{\theta_{ \frac{n}{2} }}, 
\newline
N_0 - \mathbf{\theta_{ \frac{n-2}{2} }},
\mathbf{\omega_{ \frac{n-2}{2} }}, 
\mathbf{\theta_{ \frac{n+2}{2}  }}, 
N_0 - \mathbf{\theta_{ \frac{n-4}{2}  }}, 
\dots,  
\mathbf{\omega_1},  
\mathbf{\theta_{n-1}},
N_0 - \mathbf{\theta_0},
b$
\end{center}

\end{lemma}

\begin{proof}
Observe that the packed sequences $\mathbf{\omega_j}$ and $\mathbf{\theta_{n-j}}$ are consecutive in $X_{Y_n} \cap [a,b]$, so are their complementary packed sequences. The rest of the claim follows the combination of Lemma~\ref{SYn} and Lemma~\ref{complementary}.
\end{proof}

\begin{proposition}
\label{monotonesubroot}
For $\{Y_n\}^\infty_{n=1}$, the reduced delta sequences $\tilde{\Delta}_{Y_n}$ on $[a,b]$ are respectively given following sequences.  When $n$ is odd, we have

\begin{center}
$ \left \langle n,-(n-1),n-1,\ldots, \frac{-(n+1)}{2} , \frac{n+1}{2} ,\frac{-(n+1)}{2} , \frac{n+1}{2}, \ldots, -(n-1), n-1, -n,1 \right \rangle.$
\end{center}

\noindent For even values of $n$, we also have

\begin{center}
$\left \langle n,-(n-1),n-1,\ldots,\frac{-(n+2)}{2} , \frac{n+2}{2}, \frac{-n}{2}, \frac{n}{2} ,\frac{-(n+2)}{2} , \frac{n+2}{2}, \ldots, -(n-1), n-1, -n,1 \right \rangle.$
\end{center}

\noindent Therefore, the corresponding minimal graded subroots $R_{Y_n}$ and monotone subroots $M_{Y_n}$ of $\{Y_n\}^\infty_{n=1}$ are shown in Figure~\ref{fig:rootsofYn}.
\end{proposition}

\begin{proof}

Recall that the sequences $\mathbf{\theta_k}$ and $N_0 - \mathbf{\theta_k}$ both contains $n-k$ elements and $\mathbf{\omega_\ell}$ and $N_0 - \mathbf{\omega_\ell}$ both consist of $n-2\ell$ elements. So by Lemma \ref{XYn}, $Y_n$ has the reduced delta sequences as desired. Then the corresponding graded root $R_{Y_n}$ is constructed by passing to $\tau$ sequences and using ingredients in Section~\ref{graded}. Finally, the monotone subroot $M_{Y_n}$ is extracted from $R_{Y_n}$ by following Section~\ref{monotoneroot} and applying Theorem~\ref{parametrization}.
\end{proof}

\begin{figure}[htbp]
\begin{center}
\includegraphics[width=0.7\columnwidth]{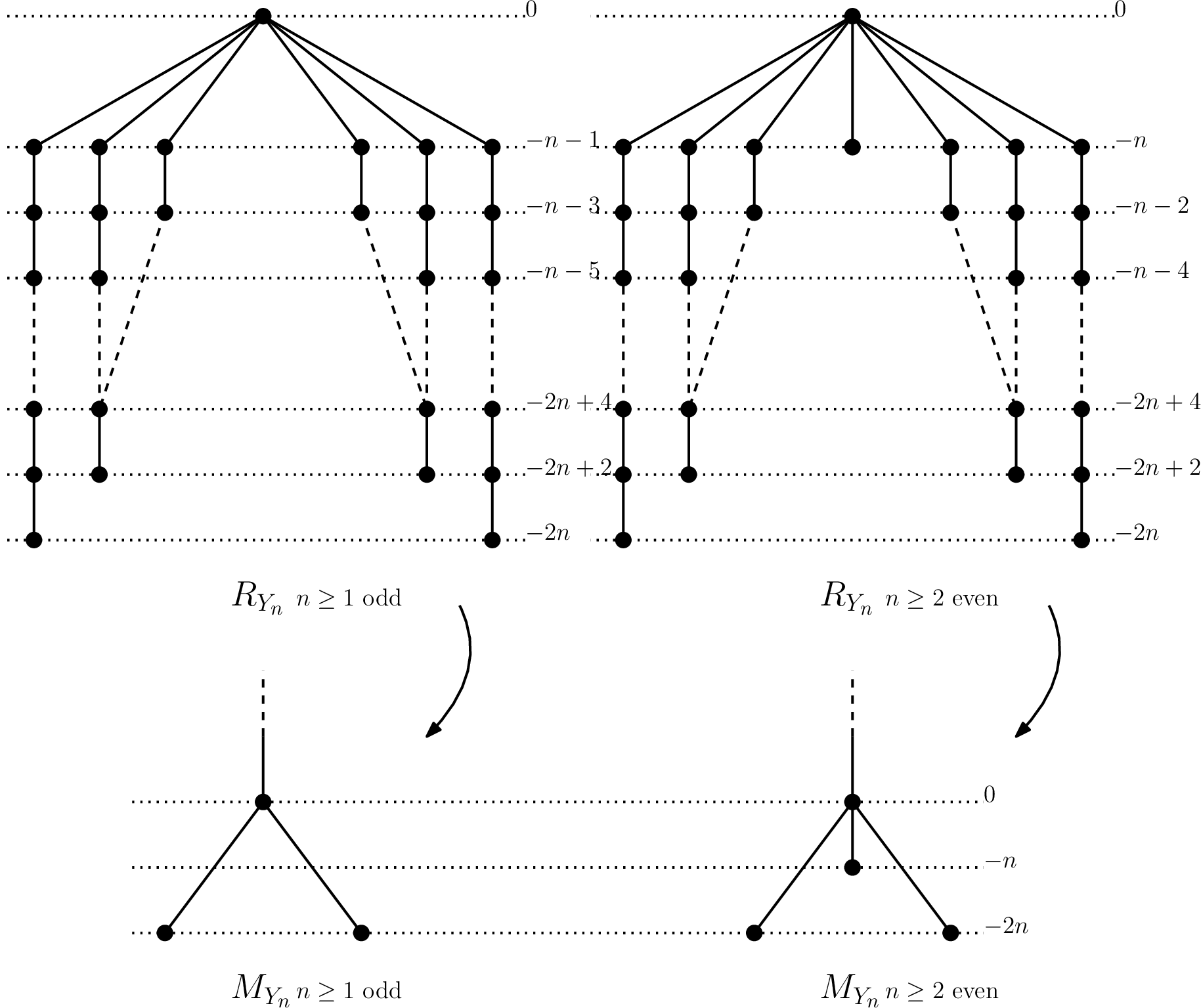}
\end{center}
\caption{The minimal graded and monotone subroots of $\{Y_n\}^\infty_{n=1}$.}
\label{fig:rootsofYn}
\end{figure}

\subsection{Computations for $\{Z_n=\Sigma(2n+1,3n+1, 6n+5)\}^\infty_{n=1}$}

We know that the lattice point realizing the $d$-invariant of $\{Z_n\}^\infty_{n=1}$ is same with one of $\{X_n\}^\infty_{n=1}$ or $\{Y_n\}^\infty_{n=1}$, by \cite[Lemma~4.6, Theorem~1.2]{KS20}, the values of $a$ and $b$ remain same except the values of $x,y$ and $z$. Note that $d(Z_n) = -2n$ for $n\geq 1$.

\begin{lemma}
\label{ab3}
For $\{Z_n\}^\infty_{n=1}$, the positions of terminal cohomology classes $k_1$ and $k_2$ of good full paths for $\{Z_n\}^\infty_{n=1}$ in the Laufer sequence are
\begin{align*}
a=x+(n-1)z \ \ \text{and} \ \ b=y+(n-1)z. 
\end{align*}
Thus $a,b \in S_{Y_n}.$
\end{lemma}

\begin{proof}
For for $\{Z_n\}^\infty_{n=1}$, we remark that $x=(2n+1)(3n+1)$, $y=(2n+1)(6n+5)$ and $z=(3n+1)(6n+5)$ . The proof of lemma is identical with Lemma~\ref{ab} or Lemma~\ref{ab2} by respecting new values of $x,y$ and $z$. 
\end{proof}

We introduce two packed sequences for $\{Z_n\}^\infty_{n=1}$:
\begin{align*}
\mathbf{\eta_k}&= \mathbf{[(2+3k)x + (n-1-k)z]} \text{ for all } k \in \{ 0,1, \ldots, n-1 \}, \\
\mathbf{\xi_\ell}&=\mathbf{[(-1+3\ell)y + (n-2\ell)z]} \text{ for all } \ell \in \left \{ 1,2, \ldots,\left \lfloor \frac{n}{2} \right \rfloor \right \}.
\end{align*}

Lemma~\ref{consecutive} indicates that $\mathbf{\eta_k}$ contains $n-k$ elements and  $\mathbf{\xi_\ell}$ contains $n+1-2\ell$ elements.

\begin{lemma}
\label{characterization3} 
For the packed sequences $\mathbf{\eta_k}$ and $\mathbf{\xi_\ell}$, the following facts hold:
	\begin{enumerate}
			\item\label{part:eta}$\mathbf{\eta_{k_2}}<\mathbf{\eta_{k_1}}$  for all $k_1,k_2\in \{0,1,\dots,n-1\}$ with $k_1<k_2$.
		\item \label{part:xi} $\mathbf{\xi_{\ell_1}}<\mathbf{\xi_{\ell_2}}$ for all $\ell_1,\ell_2\in  \left \{ 1,2, \dots,\left \lfloor \frac{n}{2} \right  \rfloor \right \}$ with $\ell_1<\ell_2$.
		\item  $a < \mathbf{\eta_k} < b$ for all $k\in \{0,1,\dots,n-1\}$. 
		\item $a< \mathbf{\xi_\ell} < b$ for all  $ \ell \in \left \{ 1,2, \dots,\left \lfloor \frac{n}{2} \right  \rfloor \right \}$.
		\item $\mathbf{\xi_{\ell}}< \mathbf{\eta_{n-\ell}}<\mathbf{\xi_{\ell+1}}$ for all $ \ell \in \left \{ 1,2, \dots, \left \lfloor \frac{n}{2} \right  \rfloor -1 \right \}$.
		\end{enumerate}
\end{lemma}

\begin{proof}
We consider the proof case by case.
\begin{enumerate}
\item Suppose $n\geq 2$ since there is nothing to show when $n=1$. It is enough to prove that $\mathbf{\eta_{k+1}}<\mathbf{\eta_{k}}$ for all $k\in \{0,1,\dots,n-2\}$. To verify the latter inequality, it suffices to compare the initial elements of the respective packed sequences, i.e., $ (2+3(k+1))x+(n-2-k)z < (2+3k)x+(n-1-k)z$. This is true because $z-3x=6n+2>0$ for all $n$. 
			\vspace{0.2cm}
		
		\item The claim is straightforward when $n\leq 2$, thus assume $n\geq 3$. It suffices to prove that $\mathbf{\xi_\ell}< \mathbf{\xi_{\ell+1}}$ for all $\ell \in \left \{ 1,2,\dots,\lfloor \frac{n}{2} \rfloor \right \}$. To show this inequality, we again compare the initial elements and show $(-1+3\ell )y+(n-2\ell)z<(-1+3(\ell+1))y+(n-2(\ell+1))z$. This inequality is equivalent to $2z-3y<0$ which holds for all $n$ since $3y-2z=6n+5$. 
		
		\vspace{0.2cm}
		
		\item In part \eqref{part:eta}, we see that $\mathbf{\eta_k}$ is decreasing. Thus, it is enough to compare the initial element of $\mathbf{\eta_{n-1}}$ with $a$ and the initial element of $\mathbf{\eta_0}$ with $b$. For the former inequality, we need to show that $x+(n-1)z < (2+3(n-1))x$ which is equivalent to $n(2x-y) + y + n-1 >0$ by the identity \eqref{xyz}. As $2x-y = -(6n+3)$ and $y=12n^2 + 16n +5$ this inequality holds for all $n$. To prove the rest, the required inequality is $2x+(n-1)z< y+(n-1)z$, which is clearly true since $y-2x=6n+3 >0$ for all $n$.

		\vspace{0.2cm}

		\item The part \eqref{part:xi} indicates that $\mathbf{\xi_\ell}$ is increasing, so we follow the strategy of the previous part by showing that two inequalities $a < \mathbf{\xi_1}$ and $\mathbf{\xi_{\left \lfloor \frac{n}{2}  \right \rfloor}} <b$ hold. Again, we compare the initial elements of required packed sequences. To verify the former inequality, we need to see that $x+(n-1)z < 2y+(n-2)z$. It reduces to $2x-y-1 < 0$ by using the identity \eqref{xyz}, which is true for all $n$ since $2x-y=-(6n+3)$. To complete the proof of this case, now we need to prove that $b$ is greater than the initial element of $\mathbf{\xi_{ \left \lfloor \frac{n}{2} \right  \rfloor}}$. The wanted inequality is $\left (-1+ 3 \left \lfloor \frac{n}{2}  \right \rfloor \right )y+\left (n - 2 \left \lfloor \frac{n}{2} \right \rfloor \right )z < y + (n-1)z$. We rewrite it as $2y-z + \left \lfloor \frac{n}{2} \right \rfloor (2z-3y) > 0$. This is true for all $n$ since $2z-3y=-(6n+5)$ and $2y-z = (6n+5)(n+1)$.
		
		\vspace{0.2cm}
		
		\item We compare the initial elements of packed sequences $\mathbf{\xi_{\ell}}, \mathbf{\eta_{n-\ell}}$ and $\mathbf{\xi_{\ell+1}}$ by showing that the inequalities $(-1+3\ell)y + (n- 2\ell ) z < (2+3(n- \ell) )x + (\ell -1)z$ and $(2+3(n- \ell) )x + (\ell -1)z < (-1+3(\ell + 1))y + (n- 2(\ell +1)) z$ both hold. Using the identity \eqref{xyz}, we reduce the former inequality $n(2x-y) + x + n - 3\ell +1 > 0$. Since $2x-y = -(6n+3)$ and $x = 6n^2 +5n +1$, it turns out that $3n-3\ell +2 > 0$ which is true for all $n$. Again by the identity \eqref{xyz}, we transform the latter inequality to $n(y-2x) + y - 3x - n +3\ell +1>0$. Since $y-2x=6n+3$ and $y-3x = -6n^2+n+2$, it is equivalent to $3n+3\ell +3 >0$. It clearly holds for all $n$. 

\end{enumerate}
\end{proof}

\begin{lemma}
\label{characterization4}
Any  element in $S_{Z_n} \cap (a,b)$ must belong to one of the packed sequences $\mathbf{\eta_k}$ for some $ k \in \{ 0,1, \ldots, n-1 \}$ or $\mathbf{\xi_\ell}$ for some $ \ell \in \left \{ 1,2, \ldots, \lfloor \frac{n}{2} \rfloor \right \}$.
\end{lemma}

\begin{proof}

We suppose for a contradiction that there exists another element $e \in S_{Z_n} \cap (a,b)$. Assume without loss of generality that $e$ is an initial element of a packed sequence. We split the proof into two cases.

\textbf{Case I.} $e=fx+hz$ where  $0 \leq h \leq n-1$.

The integer $e$ can be re-parametrized as $e=fx+(n-1-k)z$ for  $k\in\{0,1,\dots,n-1\}$. In this case, we will verify that $e$ must be the initial element of $\mathbf{\eta_k}$. By Lemma~\ref{characterization3}, we have $a < e <b$, so
\begin{align*}
x+(n-1)z < fx+(n-1-k)z<y+(n-1)z,
\end{align*} 
\noindent which is equivalent to
\begin{align*}
0 < (f-1)x-kz<y-x.
\end{align*}
\noindent Dividing all sides by $x$, we obtain
\begin{align*}
0 < f-1-k\frac{z}{x}<\frac{y}{x}-1.
\end{align*}
Notice that $\frac{y}{x}<3 $ for all $n$, thus
\begin{align*}
0 < f-1-k\frac{z}{x}<2.
\end{align*}

\noindent Here, $\frac{z}{x}=3+\frac{2}{2n+1}$ which implies that
\begin{align*}
\frac{2k}{2n+1} < f-1-3k<2+ \frac{2k}{2n+1}.
\end{align*}
Since $0\leq \frac{2k}{2n+1}<1$, we find that
\begin{align*}
 0<f-1-3k<2
\end{align*}
Clearly, the only integer solution satisfying the above inequality is $f=2+3k$ as we desired.

\textbf{Case II.} $e=gy +hz$ where $1 \leq h \leq n$.

Likewise in Case I, we show that $e$ is the initial element of the packed sequence $\mathbf{\omega_\ell}$ for some $\ell$. Write $h=n-k$ for $k\in \{1,2, \dots,n\}$.  By Lemma~\ref{characterization}, we know tht $a<e<b$, thus
\begin{align*}
x+(n-1)z<gy+(n-k)z<y+(n-1)z.
\end{align*}
Subtract $y+(n-1)z$ from all sides to get
\begin{align*}
x-y<(g-1)y +(1-k)z<0.
\end{align*}
Dividing by $y$, we see that
\begin{align*}
\frac{x}{y}-1<g-1 +(1-k) \frac{z}{y} < 0.
\end{align*}
Now plug in $\frac{x}{y}=\frac{3n+1}{6n+5}=\frac{1}{2}-\frac{3}{12n+10}$ and $\frac{z}{y}=\frac{3n+1}{2n+1}=\frac{3}{2}-\frac{1}{4n+2}$ and re-arrange the terms so that
\begin{align*}
-2-\frac{3}{12n+10}+ \frac{1}{4n+2} - \frac{k}{4n+2}<g-1-\frac{3k}{2}<-\frac{3}{2} + \frac{1}{4n+2} - \frac{k}{4n+2}
\end{align*}
Since $0\leq k \leq n$, we find that
\begin{align*}
-\frac{9}{4}<g-1-\frac{3k}{2}< -\frac{7}{4}
\end{align*}

When $k$ is odd, clearly, there is no integer solution for $g$. However, when $k$ is even, we may write $k=2\ell$ for $\ell \in \left \{1,2,\dots,\lfloor\frac{n}{2}\rfloor \right \}$ and thus we see that $g=-1+3 \ell$ as we wanted. 

\end{proof}

Following the notation of previous subsection, set $\mathbf{\theta_0 = [x+(n-1)z]}$ with the initial element $a$. Now we can precisely list all elements lying in $S_{Z_n} \cap [a,b]$.

\begin{lemma}
\label{SZn}
As an ordered set, $S_{Z_n} \cap [a,b]$ is given by the following sequences and elements \begin{center}

$\mathbf{\theta_0}, \mathbf{\xi_1}, \mathbf{\eta_{n-1}}, \mathbf{\xi_2}, \mathbf{\eta_{n-2}}, \ldots, \mathbf{\xi_{\left \lfloor \frac{n}{2} \right  \rfloor }}, \mathbf{\eta_{\left \lfloor \frac{n+1}{2} \right  \rfloor }}, \mathbf{\eta_{\left \lfloor \frac{n+1}{2} \right  \rfloor -1 }}, \ldots, \mathbf{\eta_{2}}, \mathbf{\eta_{1}}, \mathbf{\eta_{0}}, b$

\end{center}

\end{lemma}

\begin{proof}
The proof follows the combination of Lemma~\ref{ab}, Lemma~\ref{consecutive}, Lemma~\ref{characterization3}, and Lemma~\ref{characterization4}.
\end{proof}

Recall that $N_0 = pqr - x - y -z$. In the statement of next lemma, we abuse of notation by writing \emph{complementary} packed sequences as $N_0 - \mathbf{[fx + gy + hz]} =N_0-(f+h)x-(g+h)y, N_0-(f+h-1)x-(g+h-1)y-z,\dots,N_0-fx-gy-hz.$

\begin{lemma}
\label{complementary2} 
The complementary packed sequences satisfy the following inequalities:
	\begin{enumerate}
		\item $N_0 - \mathbf{\eta_{k-1}}<\mathbf{\eta_{n-k}}< N_0 - \mathbf{\eta_k}$ for all $k\in \{1,2, \dots,n-1 \}$,
		\item $N_0 - \mathbf{\eta_{\ell -1}}< \mathbf{\xi_{\ell }}<N_0 - \mathbf{\eta_\ell }$ for all $\ell \in \left \{1,2,\dots,\lfloor \frac{n}{2} \rfloor \right \},$
	\end{enumerate}
	
\end{lemma}

\begin{proof}
Plugging $p=2n+1$ and using the identity \eqref{xyz}, we see that
\begin{align}
\label{N02}
N_0 = (2n-1)x + (2n-1)y -2n
\end{align}

\begin{enumerate}
\item To see that the inequalities both hold, we compare the initial elements of given packed sequences. We prove that the inequalities $N_0 - (-1+3k)x + (n-k)z < (2+3(n-k))x + (k-1)z$ and $(2+3(n-k))x + (k-1)z < N_0 - (2+3k)x - (n-1-k)z$ are true for all $n$. We verify the first one as follows:
\begin{align*}
& (2+3(n-k))x + (k-1)z + (-1+3k)x + (n-k)z - N_0 \\
& = (1+3n)x + (n-1)z - N_0 \\
& {\overset{\eqref{xyz}} =} (1+3n)x + (n-1)x + (n-1)y -n+1 -N_0 \\
& {\overset{\eqref{N02}} =} 4nx + (n-1)y -n+1 - (2n-1)x - (2n-1)y +2n \\
& = (2n+1)x -ny +n+1 \\
& = (2x-y)n +x +n+1 >0
\end{align*}
 
\noindent This is true since $2x-y = -(6n+3)$ and $x = 6n^2 + 5n +1$. We can validate the second inequality since $y-2x = 6n+3$ and $y-3x = -6n^2 +n +2$:
\begin{align*}
& N_0 - (2+3k)x - (n-1-k)z - (2-3(n-k))x - (k-1)z \\
& = N_0 - (3n+4)x - (n-2)z \\
& {\overset{\eqref{xyz}} =} (2n-1)x + (2n-1)y -2n - (3n+4)x - (n-2)z \\
& {\overset{\eqref{N02}} =} (-n-5)x + (2n-1)y -2n - (n-2)x - (n-2)y +n-2 \\
& = -(2n+3)x + (n+1)y - n -2 \\
&= (y-2x)n +y-3x -n -2 >0
\end{align*}

			\vspace{0.2cm}
\item By comparing the initial elements of our sequences, we will deal with this case. It is enough to show that the inequalities $N_0 - (2+3(\ell -1))x -(n-\ell )z < (-1+3\ell )y + (n-2\ell )z$ and $(-1+3\ell )y + (n-2\ell )z < N_0 - (2+3\ell )x -(n-1-\ell )z$ hold for all $n$. Doing the following computations, we confirm that the first inequality holds:
\begin{align*}
& (-1+3\ell )y + (n-2\ell )z + (2+3(\ell -1))x +(n-\ell )z - N_0 \\
& = (-1+3\ell )y + (2n-3\ell )z + (2+3(\ell -1))x - N_0 \\
& {\overset{\eqref{xyz}} =}  (-1+3\ell )y + (2n-3\ell )x + (2n-3\ell )y -2n +3\ell + (2+3(\ell -1))x - N_0 \\
& {\overset{\eqref{N02}} =} (2n-1)x + (2n-1)y -2n +3\ell - (2n-1)x - (2n-1)y +2n \\
& = 3\ell > 0
\end{align*}

\noindent Finally, we guarantee the validity of the last inequality in the following way since $y-2x = 6n+3$:
\begin{align*}
& N_0 - (2+3\ell )x -(n-1-\ell )z - (-1+3\ell )y - (n-2\ell )z \\
& = N_0 - (2+3\ell )x - (-1+3\ell )y - (2n-1-3\ell )z \\
& {\overset{\eqref{xyz}} =} N_0 - (2+3\ell )x - (-1+3\ell )y - (2n-1-3\ell )x - (2n-1-3\ell )y - 2n+1+3\ell  \\
& {\overset{\eqref{N02}} =} (2n-1)x + (2n-1)y -2n - (2n+1)x - (2n-2)y + 2n - 3\ell -1 \\
& = y-2x - 3\ell +1 > 0
\end{align*}

\end{enumerate}
\end{proof}

\begin{lemma}
\label{XZn}
As an ordered set, the elements of $X_{Z_n} \cap [a,b]$ is given by the following sequences and elements. When $n$ is odd, we have 

\begin{center}
$\mathbf{\theta_0}, 
N_0-\mathbf{\eta_{0}},
\mathbf{\xi_1}, 
\mathbf{\eta_{n-1}}, 
\ldots, 
N_0 - \mathbf{\eta_{ \frac{n-3}{2} }}, 
\mathbf{\xi_{ \frac{n-1}{2}  }}, 
\mathbf{\eta_{ \frac{n+1}{2} }}, 
N_0 - \mathbf{\eta_{ \frac{n-1}{2} }}, 
\mathbf{\eta_{ \frac{n-1}{2} }}, 
N_0 - \mathbf{\eta_{ \frac{n+1}{2} }}, 
\newline
N_0 - \mathbf{\xi_{ \frac{n-1}{2}  }}, 
\mathbf{\eta_{ \frac{n-3}{2}}}, 
\ldots, 
N_0 - \mathbf{\eta_{n-1}}, 
N_0 - \mathbf{\xi_1}, 
\mathbf{\eta_{0}},
N_0 - \mathbf{\theta_0}, 
b$
\end{center}

\noindent For even values of $n$, we also have

\begin{center}
$\mathbf{\theta_0}, 
N_0-\mathbf{\eta_{0}},
\mathbf{\xi_1}, 
\mathbf{\eta_{n-1}}, 
\ldots, 
N_0 - \mathbf{\eta_{ \frac{n-2}{2} }}, 
\mathbf{\xi_{ \frac{n}{2} }}, 
\mathbf{\eta_{ \frac{n}{2} }}, 
N_0 - \mathbf{\eta_{ \frac{n}{2} }}, 
N_0 - \mathbf{\xi_{ \frac{n}{2} }}, 
\newline
\mathbf{\eta_{ \frac{n-2}{2} }}, 
\ldots, 
N_0 - \mathbf{\eta_{n-1}}, 
N_0 - \mathbf{\xi_1}, 
\mathbf{\eta_{0}},
N_0 - \mathbf{\theta_0}, 
b$
\end{center}

\end{lemma}

\begin{proof}
Clearly complementary packed sequences of $\mathbf{\xi_j}$ and $\mathbf{\eta_{n-j}}$ are consecutive in $X_{Z_n} \cap [a,b]$ since $\mathbf{\xi_j}$ and $\mathbf{\eta_{n-j}}$ are consecutive. The order amongst sequences and the element $b$ is due to Lemma~\ref{SZn} and Lemma~\ref{complementary2}.
\end{proof}

\begin{proposition}
\label{monotonesubroot2}
For $\{Z_n\}^\infty_{n=1}$, the reduced delta sequences $\tilde{\Delta}_{Z_n}$ on $[a,b]$ are respectively given following sequences.  When $n$ is odd, we have

\begin{center}
$ \left \langle n,-n,n,\ldots,\frac{-(n+3)}{2} , \frac{n+3}{2}, \frac{-(n+1)}{2} , \frac{n+1}{2}, \frac{-(n+3)}{2} , \frac{n+3}{2}, \ldots, -n, n, -n,1 \right \rangle.$
\end{center}

\noindent For even values of $n$, we also have

\begin{center}
$\left \langle n,-n,n,\ldots,\frac{-(n+2)}{2}, \frac{n+2}{2} , \frac{-(n+2)}{2}, \frac{n+2}{2}, \ldots, n, n, -n,1 \right \rangle.$
\end{center}

\noindent Therefore, the corresponding minimal graded subroots $R_{Z_n}$ and monotone subroots $M_{Z_n}$ of $\{Z_n\}^\infty_{n=1}$ are shown in Figure~\ref{fig:rootsofZn}.
\end{proposition}

\begin{proof}
Recap that the sequence $\mathbf{\theta_0}$ has $n$ elements. Similarly, $\mathbf{\eta_k}$ and $N_0 - \mathbf{\eta_k}$ both consist of $n-k$ elements, $\mathbf{\xi_\ell}$ and $N_0 - \mathbf{\xi_\ell}$ both contains $n+1-2\ell$ elements. Then $Z_n$ has the reduced delta sequences by Lemma \ref{XZn}. The minimal graded root $R_{Z_n}$ and monotone subroot $M_{Z_n}$ corresponding to $\tilde{\Delta}_{Z_n}$ can be extracted by using the recipe in Section~\ref{graded}, Section~\ref{monotoneroot} and Theorem~\ref{parametrization}.
\end{proof}

\begin{figure}[htbp]
\begin{center}
\includegraphics[width=0.7\columnwidth]{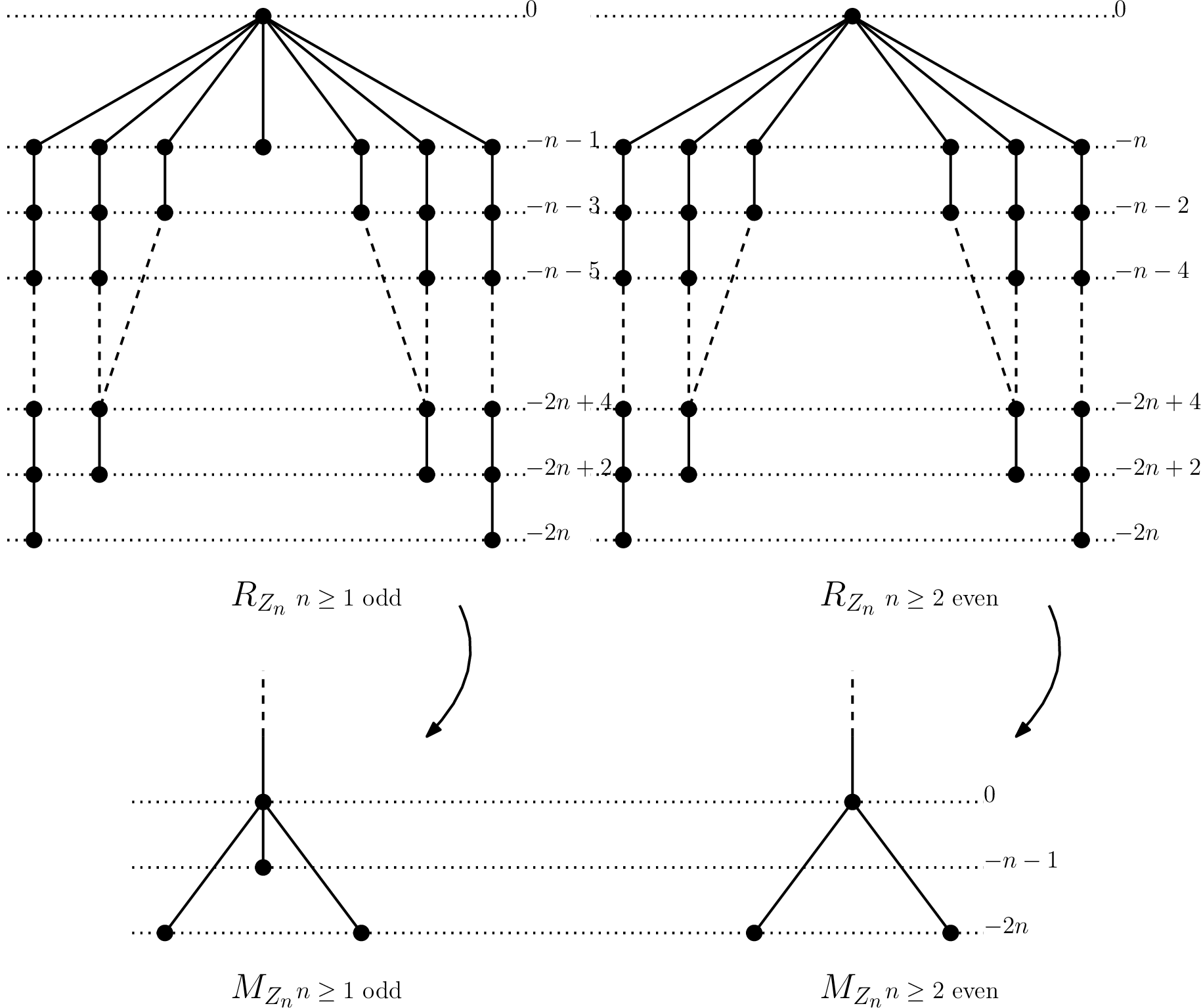}
\end{center}
\caption{The minimal graded and monotone subroots of $\{Z_n\}^\infty_{n=1}$.}
\label{fig:rootsofZn}
\end{figure}

\begin{remark}
\label{comparison}
For the families $\{ Y_n \}^\infty_{n=1}$ and $\{ Z_n \}^\infty_{n=1}$, our computation experiments indicate that the ratio between the number of leaves of minimal graded subroots and whole graded roots goes to zero as $n \to \infty$, see Figure~\ref{fig:minrootsofYn} and Figure~\ref{fig:minrootsofZn}.
\end{remark}

Finally, we are able to prove our theorems stated in the introduction.

\begin{proof}[Proof of Theorem~\ref{main}]
The connected Heegaard Floer homology is defined over the minus flavor of Heegaard Floer homology, see Theorem~\ref{connectedHF} or \cite[Definition~3.2]{HHL18}. On the homology level, the transition from the plus flavour of Heegaard Floer homology to the negative one can be done by changing the grading with minus version and subtracting two, \cite[Proposition~7.11]{OS06}. Using N\'emethi's isomorphism in \eqref{nemethi}, this is also true for the lattice homology up to the relevant grading shift.

We compute monotone roots of $\{Y_n\}^\infty_{n=1}$ and $\{Z_n\}^\infty_{n=1}$ in Proposition~\ref{monotonesubroot} and Proposition~\ref{monotonesubroot2}, respectively. In turn they are displayed in Figure~\ref{fig:rootsofYn} and Figure~\ref{fig:rootsofZn}. One can apparently see the transition described above by comparing Figure~\ref{fig:monotoneyz} with Figure~\ref{fig:rootsofYn} and Figure~\ref{fig:rootsofZn}. Using Theorem~\ref{connectedHF} and Corollary~\ref{connectedHF2} of Hendricks, Hom and Lidman, we complete the proof, see Table~\ref{tab:chfl} for the exposition.

\end{proof}

\begin{proof}[Proof of Theorem~\ref{cobordism}]

The monotone subroots of $\{ Y_n \}_{n\in \mathbb{N} }$ and $\{ Z_n \}_{n\in \mathbb{N} }$ are shown in Figure~\ref{fig:rootsofYn} and Figure~\ref{fig:rootsofZn}, respectively. Then we list the invariants of $\Theta^3_{\mathbb{Z}}$ corresponding to our families in Table~\ref{tab:cobordism}.

\begin{table}[htbp]
\begin{tabular}{|l|l|l|l|l|l|}
\hline
\textbf{$d$-invariant} & $n \geq 1$ odd & $n \geq 2$ even & \textbf{$\bar{\mu}$-invariant} & $n \geq 1$ odd & $n \geq 2$ even\\ \hline
$[Y_n]$ & $-2n$ & $-2n$ & $[Y_n]$ & $0$ & $-\frac{n}{2}$ \\ \hline
$[Z_n]$ & $-2n$ & $-2n$ & $[Z_n]$ & $\frac{-n-1}{2}$ & $0$ \\ \hline
\end{tabular}
\vskip\baselineskip
\caption{The homology cobordism invariants of $\{ Y_n \}_{n\in \mathbb{N} }$ and $\{ Z_n \}_{n\in \mathbb{N} }$.}
\label{tab:cobordism}
\end{table}
Applying Theorem~\ref{parametrization}, now we are able to parametrize monotone subroots of $\{ Y_n \}^\infty_{n=1}$ and $\{ Z_n \}^\infty_{n=1}$ under the homomorphism $h$ of Hendricks, Manolescu and Zemke in \eqref{hm}. When $n$ is odd, we have $$h([Y_n]) = M(2n,0).$$ For even values of $n$, we find 
\begin{align*}
h([Y_n]) &= M(2n,0;n,n), \\
&=M(2n,0)+M(n,n)-M(n,0), \\
&=(M(2n,0)-M(n,0))[-n].
\end{align*} 

\noindent Similarly, if $n$ is odd, then we obtain
\begin{align*}
h([Z_n]) &= M(2n,0;n+1,n+1), \\
&=M(2n,0)+M(n+1,n+1)-M(n+1,0), \\ 
&=(M(2n,0)-M(n+1,0))[-n-1].
\end{align*}

\noindent When $n$ is even, we have $$h([Z_n]) = M(2n,0).$$

Under the forgetful homomorphism from ${\mathfrak{I}}$ to $\widehat{\mathfrak{I}}$, we find that \[ \widehat{h}([Y_n]) = \begin{cases} 
      \mathcal{C}(-,n), & n \ \text{is odd}, \\
      \mathcal{C}(-,n) - \mathcal{C}(-,n/2), & n \ \text{is even}.
   \end{cases}
 \] and \[ \widehat{h}([Z_n]) = \begin{cases} 
      \mathcal{C}(-,n) - \mathcal{C}(-,(n+1)/2), & n \ \text{is odd}, \\
      \mathcal{C}(-,n), & n \ \text{is even}.
   \end{cases}
\]

Next we use the family of homomorphisms $\{ \phi_k \}_{k \in \mathbb{N}}$ in \eqref{phi} and we get
\begin{align*}
\phi_k(\mathcal{C}(-,n)) &= \delta_{n,k}, \\ 
\phi_k(\mathcal{C}(-,n/2)) &= \delta_{n/2,k}, \\ 
\phi_k(\mathcal{C}(-,(n+1)/2)) &= \delta_{(n+1)/2,k}.
\end{align*}
where $\delta$ denotes Kronecker delta function. Therefore, the remaining claim follows from Theorem~\ref{count} of Dai, Hom, Stoffregen and Truong.

Finally we compare homology cobordism classes of our families with the family of Dai, Hom, Stoffregen and Truong.

From the discussion of Floer theoretic invariants, we see that $\{X_n\}^\infty_{n=1}$ is not homology cobordant $\{ Y_n \}^\infty_{n=1}$ for even values of $n$. When $n$ is odd, $\{ X_n \}^\infty_{n=1}$ and $\{ Z_n \}^\infty_{n=1}$ are not homology cobordant as well. Moreover, $\{ Y_n \}^\infty_{n=1}$ and $\{ Z_n \}^\infty_{n=1}$ are not homology cobordant for every $n$.

Note that the ASL graph has the central vertex $-2$, see Figure~\ref{fig:plumb}. Due to the short-cut of Neumann and Zagier \cite{NZ85}, $R$-invariant of Fintushel and Stern \cite{FS85} of these families can be directly computed as follows: $$R([X_n]) = R([Y_n]) = R([Z_n]) = 1 \ \ \ \text{for all} \ n\in \mathbb{N}.$$

One can easily verify that for any $n \in \mathbb{N}$ with $n>1$ we have $$(2n+1)(4n+1)(4n+3) < (2n+1)(3n+2)(6n+1),$$ and for all $n \in \mathbb{N}$ $$(2n+1)(4n+1)(4n+3) < (2n+1)(3n+1)(6n+5).$$

By \cite[Corollary~2.1]{F90}, we conclude that the family of Dai, Hom, Stoffregen and Truong is not homology cobordant to our two families with an exception $X_1 = \Sigma(3,5,7) = Y_1$. This obstruction also can be successfully done by new gauge theoretic invariants of Daemi \cite{Dae20} and Nozaki, Sato and Taniguchi \cite{NST22}. 
\end{proof} 

\begin{figure}[htbp]
\begin{center}
\includegraphics[width=1\columnwidth]{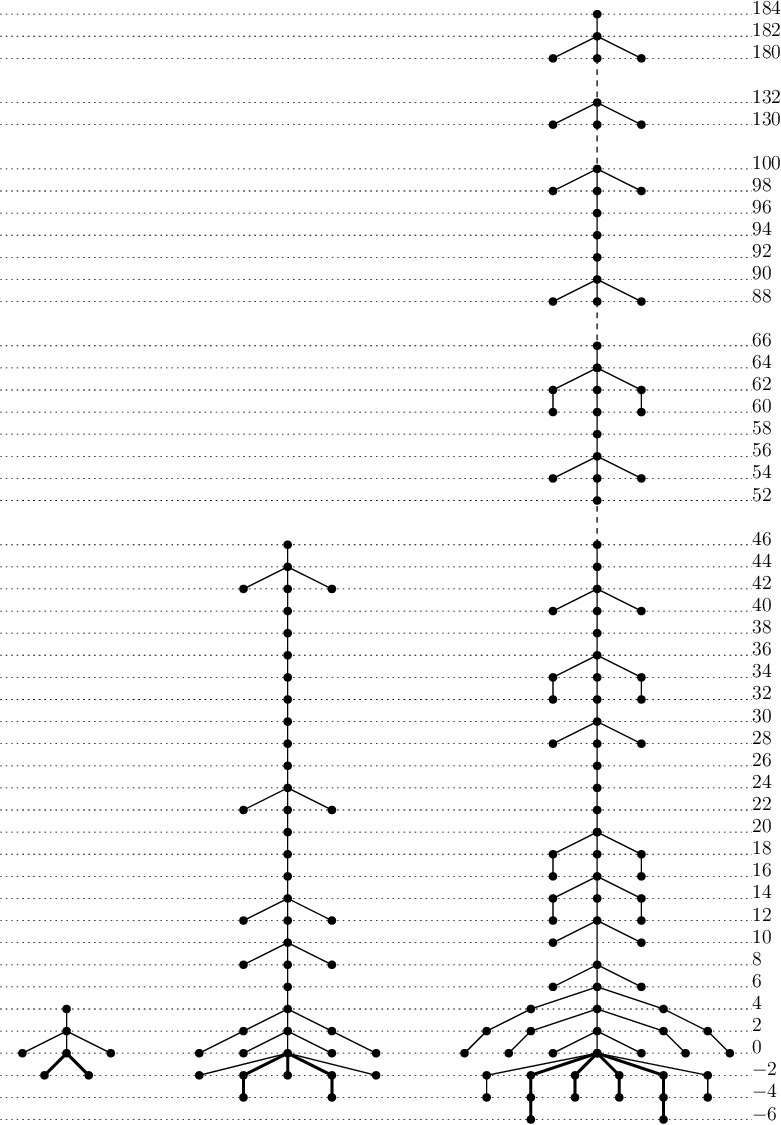}
\end{center}
\caption{The whole graded roots of $Y_1, Y_2$ and $Y_3$ with minimal graded subroots drawn in dark black.}
\label{fig:minrootsofYn}
\end{figure}

\begin{figure}[htbp]
\begin{center}
\includegraphics[width=1\columnwidth]{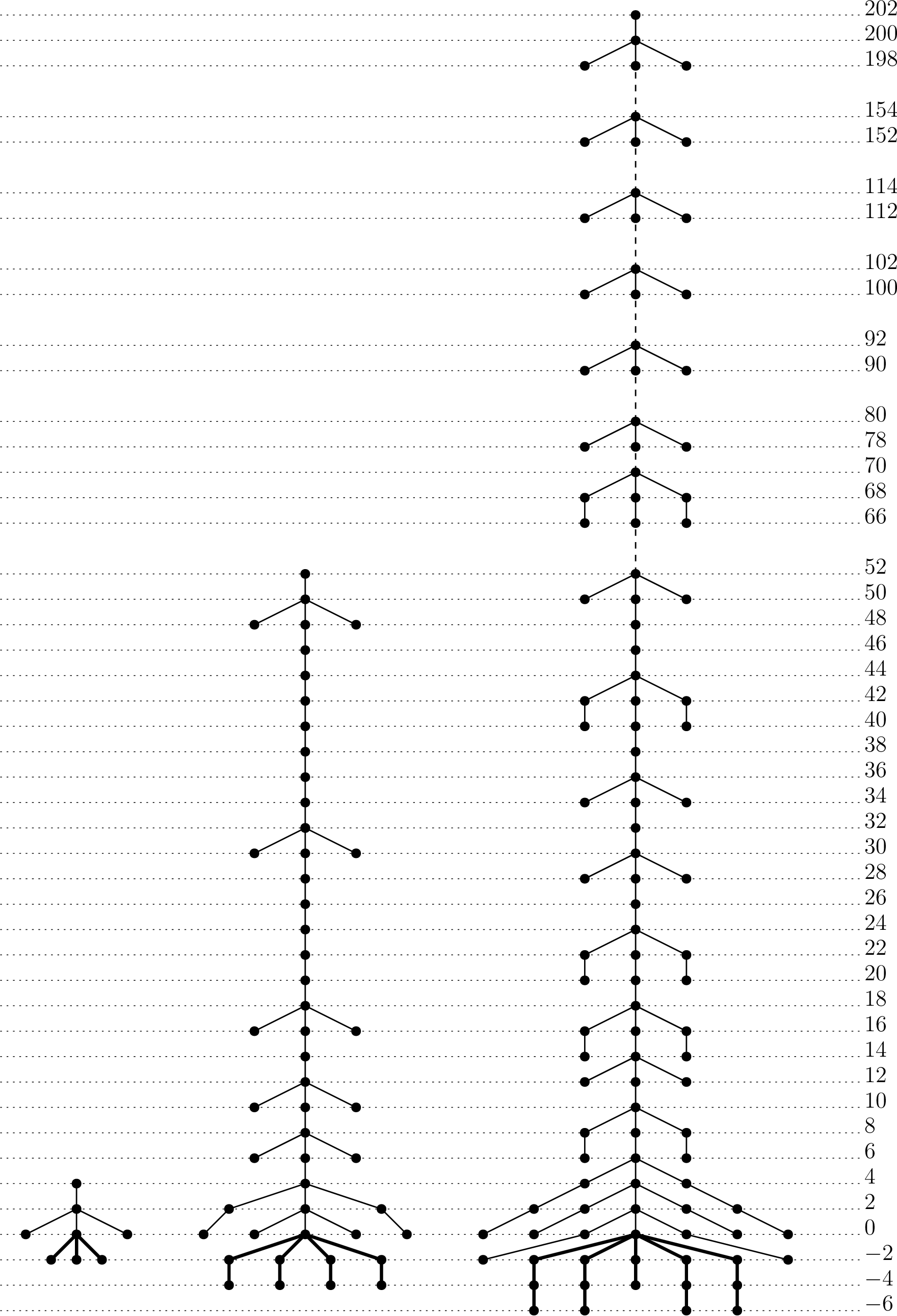}
\end{center}
\caption{The whole graded roots of $Z_1, Z_2$ and $Z_3$ with minimal graded subroots drawn in dark black.}
\label{fig:minrootsofZn}
\end{figure}
\bibliography{KarakurtSavk}
\bibliographystyle{amsalpha}

\end{document}